\documentclass[11pt]{article}

\usepackage[T1]{fontenc}
\usepackage[utf8]{inputenc}
\usepackage{lmodern}
\usepackage{amsmath,amssymb,amsfonts,amsthm,mathtools}
\usepackage{mathrsfs}
\usepackage{enumitem}
\usepackage{geometry}
\usepackage{microtype}
\usepackage[colorlinks=true,linkcolor=blue,citecolor=blue,urlcolor=blue]{hyperref}
\newcommand{\Hess}{\operatorname{Hess}}
\setlist[itemize]{leftmargin=2em}
\setlist[enumerate]{leftmargin=2em}

\newtheorem{theorem}{Theorem}[section]
\newtheorem{proposition}[theorem]{Proposition}
\newtheorem{lemma}[theorem]{Lemma}
\newtheorem{corollary}[theorem]{Corollary}
\theoremstyle{definition}
\newtheorem{definition}[theorem]{Definition}

\theoremstyle{remark}
\newtheorem{remark}[theorem]{Remark}
\theoremstyle{definition}

\newcommand{\B}{\mathbb B}
\newcommand{\C}{\mathbb C}
\newcommand{\D}{\mathbb D}
\newcommand{\T}{\mathbb T}
\newcommand{\R}{\mathbb R}

\newcommand{\Aut}{\operatorname{Aut}}
\newcommand{\barB}{\operatorname{bar}_{B}}

\newcommand{\Id}{\operatorname{Id}}
\newcommand{\dd}{\,d}

\newcommand{\Real}{\operatorname{Re}}

\title{Bergman barycentric extensions of boundary homeomorphisms of the complex ball}
\author{David Kalaj \and Anton Gjokaj \and Vladimir Ja\'cimovi\'c}
\date{}

\newcommand{\wind}{\operatorname{wind}}
\begin{document}
\maketitle

\begin{abstract}
We define a Bergman barycentric extension of boundary homeomorphisms of the
complex unit ball by
\[
        E_B(f)(z)=\barB(f_*\sigma_z),
\]
where \(\barB\) is the Busemann barycenter for the Bergman metric and
\(\sigma_z\) is the visual, equivalently Poisson--Szeg\H{o}, measure based at
\(z\).  We prove well-definedness, prescribed boundary values, full
\(\Aut(\B^n)\)-naturality, and interior real-analyticity.

For every CR-quasisymmetric boundary homeomorphism, the Bergman barycentric
extension is a quasi-isometry of complex hyperbolic space; moreover
\(E_B(f^{-1})\) is a coarse inverse of \(E_B(f)\).  Under sufficiently small
positive CR cross-ratio distortion we obtain a sharper Tukia-type theorem:
for every \(M>1\), the extension is a real-analytic diffeomorphism satisfying
\[
        M^{-1}d_B(x,y)\le d_B(E_B(f)(x),E_B(f)(y))\le M d_B(x,y).
\]
In contrast, for \(n\ge2\) there are smooth CR-orientation-preserving
CR-quasisymmetric boundary diffeomorphisms whose barycentric extensions are
non-injective and, after normalization, have singular differential.  Thus
large-scale quasi-isometric control persists on the full CR-quasisymmetric
class, whereas local non-degeneracy requires stronger boundary control.
\end{abstract}

\section{Introduction}

The classical Douady--Earle extension assigns to every orientation-preserving
homeomorphism of the unit circle a conformally natural self-map of the disk
\cite{DE}; see also \cite{Petersen}.  The construction is barycentric: the
extended point is characterized as the zero of the gradient of an averaged
Busemann potential.  This point of view suggests a natural extension problem
for the complex unit ball equipped with its Bergman metric.

Let
\[
        \B^n=\{z\in\C^n:|z|<1\},
        \qquad S^{2n-1}=\partial\B^n,
\]
and let \(d_B\) denote the Bergman distance.  For \(z\in\B^n\), let
\(\sigma_z\) be the visual probability measure based at \(z\); in ball
coordinates it is the Poisson--Szeg\H{o} measure
\[
        d\sigma_z(\xi)=
        \left(\frac{1-|z|^2}{|1-\langle z,\xi\rangle|^2}\right)^n
        d\sigma(\xi),
\]
where \(\sigma\) is normalized surface measure on \(S^{2n-1}\).  If
\(\barB(\mu)\) denotes the Busemann barycenter of a probability measure
\(\mu\) on the boundary, we define, for a homeomorphism
\(f:S^{2n-1}\to S^{2n-1}\),
\begin{equation}\label{main-extension-intro}
        E_B(f)(z)=\barB(f_*\sigma_z).
\end{equation}
The relevant naturality group is \(\Aut(\B^n)\), while the boundary geometry
is the CR, equivalently Heisenberg, geometry of \(S^{2n-1}\).  Related
conformal and holomorphic barycenters in hyperbolic balls were introduced by
Ja\'cimovi\'c and Kalaj \cite{JacimovicKalaj2025}.

The first result establishes the basic properties of the construction.

\begin{theorem}[Bergman barycentric extension]\label{main-theorem}
Let \(f:S^{2n-1}\to S^{2n-1}\) be a homeomorphism.  Then \(E_B(f)\) is
well-defined and real-analytic in \(\B^n\), and it extends continuously to
\(\overline{\B^n}\) with boundary value \(f\).  Moreover, for all
\(A,B\in\Aut(\B^n)\),
\begin{equation}\label{naturality-main}
        E_B(A\circ f\circ B)=A\circ E_B(f)\circ B.
\end{equation}
In particular, if \(f=A|_{S^{2n-1}}\) for some \(A\in\Aut(\B^n)\), then
\(E_B(f)=A\).
\end{theorem}

Our first main geometric result is non-perturbative and concerns the full
CR-quasisymmetric class.  General boundary theory for Gromov hyperbolic spaces
guarantees the existence of quasi-isometric extensions of quasisymmetric
boundary maps \cite{BonkSchramm}.  Here the point is that the canonical map
\(E_B(f)\) itself is a quasi-isometry.

\begin{theorem}[Quasi-isometry theorem]\label{theorem-general-quasi-isometry}
Let \(n\ge2\), and let
\(f:S^{2n-1}\to S^{2n-1}\) be \(\eta\)-CR-quasisymmetric.  Then
\(F=E_B(f):(\B^n,d_B)\to(\B^n,d_B)\) is a quasi-isometry.  More precisely,
there exist \(L\ge1\) and \(C\ge0\), depending only on \(n\) and \(\eta\),
such that
\begin{equation}\label{general-quasi-isometry-estimate}
        L^{-1}d_B(x,y)-C
        \le d_B(F(x),F(y))
        \le Ld_B(x,y)+C
\end{equation}
for all \(x,y\in\B^n\).  If \(G=E_B(f^{-1})\), then there exists
\(D=D(n,\eta)<\infty\) such that
\begin{equation}\label{coarse-inverse-estimates}
        d_B(G(F(x)),x)\le D,
        \qquad
        d_B(F(G(y)),y)\le D
\end{equation}
for all \(x,y\in\B^n\).  Thus \(E_B(f^{-1})\) is a coarse inverse of
\(E_B(f)\).
\end{theorem}

The preceding conclusion is deliberately large-scale.  It does not imply
injectivity, local non-degeneracy, or differential quasiconformality.  Under
small automorphism-invariant boundary distortion, however, the conclusion
becomes much stronger.  We call a CR-orientation-preserving homeomorphism
\(f\) \(\varepsilon\)-almost CR-M\"obius if its CR metric cross-ratio changes
by at most the multiplicative factor \(e^{\varepsilon}\); the precise
definition is given in Section~\ref{section-small-distortion}.

\begin{theorem}[Tukia-type almost-isometry theorem]
\label{theorem-tukia-type-almost-isometry}
Let \(n\ge2\).  For every \(M>1\) there exists
\(\varepsilon_0=\varepsilon_0(M,n)>0\) with the following property.  If
\(0<\varepsilon<\varepsilon_0\) and
\(f:S^{2n-1}\to S^{2n-1}\) is a CR-orientation-preserving
\(\varepsilon\)-almost CR-M\"obius homeomorphism, then
\(F=E_B(f):\B^n\to\B^n\) is a real-analytic diffeomorphism and
\begin{equation}\label{tukia-type-distance-estimate}
        M^{-1}d_B(x,y)
        \le d_B(F(x),F(y))
        \le M d_B(x,y)
\end{equation}
for all \(x,y\in\B^n\).  In particular, \(F\) is Bergman-quasiconformal and
\[
        K_B(F)\le M^2.
\]
\end{theorem}

The mechanism behind the perturbative theorem is a normalized first-variation
formula.  Given \(F=E_B(f)\) and \(z_0\in\B^n\), automorphism naturality allows
one to choose \(A,B\in\Aut(\B^n)\) so that
\[
        g=A\circ f\circ B|_{S^{2n-1}},
        \qquad E_B(g)(0)=0.
\]
For such normalized data we prove
\begin{equation}\label{intro-factorization}
        dE_B(g)_0=\mathcal A_g^{-1}\mathcal B_g,
\end{equation}
where
\[
        \mathcal B_gv
        =2n\int_{S^{2n-1}}
        g(\xi)\Real\langle v,\xi\rangle\,d\sigma(\xi).
\]
Small CR cross-ratio distortion implies, after normalization, uniform
closeness to a unitary map.  Formula \eqref{intro-factorization} then yields
uniform two-sided control of the Bergman singular values of the differential.

The small-distortion hypothesis cannot be removed from the local statement.
The final main result gives a counterexample inside the CR-quasisymmetric
class.

\begin{theorem}[Counterexample and sharpness]\label{theorem-cr-qs-counterexample}
Let \(n\ge2\).  There exists a CR-orientation-preserving CR-quasisymmetric
homeomorphism
\[
        h:S^{2n-1}\to S^{2n-1}
\]
such that the Bergman barycentric extension
\[
        E_B(h):\B^n\to\B^n
\]
is not injective.
\end{theorem}

After automorphic normalization, the same construction yields a
CR-orientation-preserving CR-quasisymmetric homeomorphism \(g\) with
\(E_B(g)(0)=0\) and singular differential \(dE_B(g)_0\).  Thus the
quasi-isometry theorem is genuinely a large-scale result: fixed
CR-quasisymmetric distortion controls the extension coarsely, but does not in
general control the smallest singular value of its differential.

The paper is organized as follows.  Section~2 collects the Bergman and CR
preliminaries, defines the barycentric extension, and develops the normalized
first-variation formula and automorphic normalization used throughout the
paper.  Section~3 contains the proofs of the main positive results: the basic
extension theorem, the non-perturbative quasi-isometry theorem, and the
Tukia-type almost-isometry theorem.  Section~4 constructs the contact
``tennis-ball'' counterexample, proves the failure of injectivity and local
non-collapse in the full CR-quasisymmetric class, and concludes with a brief
discussion of the large-scale/local dichotomy.

\section{Preliminaries and the Bergman barycentric extension}
\label{section-preliminaries}

\subsection{Bergman geometry and Busemann barycenters}

Let \(\B^n\subset\C^n\) be the unit ball.  We use the Bergman metric associated
with the K\"ahler potential
\[
        -\log(1-|z|^2).
\]
Up to a harmless positive constant, its Hermitian tensor is
\begin{equation}\label{bergman-tensor}
        g_{i\overline j}(z)=
        \frac{(1-|z|^2)\delta_{ij}+\overline z_i z_j}{(1-|z|^2)^2}.
\end{equation}
This metric is complete, invariant under \(\Aut(\B^n)\), and negatively
curved.  Hence \((\B^n,g_B)\) is complex hyperbolic \(n\)-space
\cite{Goldman}.

For \(\xi\in S^{2n-1}\), we use the normalized Busemann function
\begin{equation}\label{busemann-function}
        \beta_\xi(z)=
        \log\frac{|1-\langle z,\xi\rangle|^2}{1-|z|^2},
        \qquad z\in\B^n.
\end{equation}
Thus \(\beta_\xi(0)=0\).  Multiplying all Busemann functions by the same
positive constant would not change barycenters.

\begin{lemma}[Busemann covariance]\label{busemann-cocycle}
For every \(A\in\Aut(\B^n)\) there exists a function \(c(A,\xi)\), independent
of \(z\), such that
\begin{equation}\label{cocycle-formula}
        \beta_{A\xi}(Az)=\beta_\xi(z)+c(A,\xi).
\end{equation}
Consequently the gradient fields of Busemann functions transform equivariantly
under holomorphic automorphisms.
\end{lemma}

\begin{proof}
Elements of \(\Aut(\B^n)\) are Bergman isometries.  Hence the function
\(z\mapsto \beta_\xi(A^{-1}z)\) is a Busemann function with ideal endpoint
\(A\xi\).  The function \(z\mapsto\beta_{A\xi}(z)\) has the same endpoint.
Busemann functions with the same endpoint and with the same metric
normalization differ by an additive constant.  Thus
\[
        \beta_{A\xi}(Az)=\beta_\xi(z)+c(A,\xi),
\]
where the constant is independent of \(z\).  Since \(A\) preserves the
Levi--Civita connection of the Bergman metric, gradients and Hessians transform
by pull-back under \(A\); in particular, taking gradients removes the constant.
\end{proof}

Let \(\mu\) be a probability measure on \(S^{2n-1}\).  Define its averaged
Busemann potential by
\begin{equation}\label{barycenter-potential}
        \mathcal B_\mu(z)=\int_{S^{2n-1}}\beta_\xi(z)\,d\mu(\xi).
\end{equation}

\begin{definition}
A probability measure \(\mu\) on \(S^{2n-1}\) is called admissible if it has no
atom of mass at least \(1/2\).
\end{definition}

The threshold \(1/2\) is natural: an atom of mass at least \(1/2\) can destroy
properness of the averaged Busemann potential along the geodesic ray ending at
that atom.

\begin{proposition}[Existence and uniqueness]\label{existence-uniqueness-barycenter}
Let \(\mu\) be an admissible probability measure on \(S^{2n-1}\) with full
support.  Then \(\mathcal B_\mu\) is proper and has a unique minimizer in
\(\B^n\).
\end{proposition}

\begin{proof}
Busemann functions on complex hyperbolic space are convex along geodesics, and
so is their average; compare the general CAT\((0)\) discussion in
\cite{BH}.  Properness follows from the standard asymptotics of
\eqref{busemann-function}.  If \(z_j\to\eta\in S^{2n-1}\), then the Busemann
function with endpoint \(\eta\) tends to \(-\infty\) along the ray ending at
\(\eta\), whereas Busemann functions with endpoints separated from \(\eta\) tend
to \(+\infty\).  Since \(\mu\) has no atom of mass at least \(1/2\), one may
choose a neighborhood \(U\) of \(\eta\) with \(\mu(U)<1/2\).  The positive
contribution from the complement of \(U\) dominates the negative contribution
from \(U\), and hence \(\mathcal B_\mu(z_j)\to+\infty\).

Properness gives existence of a minimizer.  If two distinct minimizers existed,
convexity would force \(\mathcal B_\mu\) to be constant on the geodesic segment
joining them.  Equality in the Busemann convexity inequality would then force
\(\mu\)-almost every endpoint to lie in the two-point ideal boundary of that
geodesic.  This contradicts full support.  Hence the minimizer is unique.
\end{proof}

\begin{definition}
For an admissible full-support probability measure \(\mu\), its Bergman
barycenter is
\begin{equation}\label{definition-bar}
        \barB(\mu)=\operatorname*{argmin}_{z\in\B^n}\mathcal B_\mu(z).
\end{equation}
Equivalently,
\begin{equation}\label{gradient-equation}
        \int_{S^{2n-1}}\nabla_z^B\beta_\xi(z)\,d\mu(\xi)=0,
        \qquad z=\barB(\mu).
\end{equation}
\end{definition}

The Euler equation can be written in coordinates.  From
\eqref{busemann-function},
\[
        \frac{\partial\beta_\xi}{\partial\overline z_j}(z)
        =
        \frac{z_j}{1-|z|^2}-\frac{\xi_j}{1-\langle \xi,z\rangle},
\]
where \(\langle \xi,z\rangle=\sum_j\xi_j\overline z_j\).  Thus the barycenter
equation is
\begin{equation}\label{coordinate-barycenter-equation}
        \int_{S^{2n-1}}
        \frac{\xi}{1-\langle \xi,z\rangle}\,d\mu(\xi)
        =
        \frac{z}{1-|z|^2}.
\end{equation}
\begin{proposition}[Equivariance of barycenters]\label{equivariance-barycenter}
For every \(A\in\Aut(\B^n)\) and every admissible full-support probability
measure \(\mu\),
\begin{equation}\label{bar-equivariance}
        \barB(A_*\mu)=A(\barB(\mu)).
\end{equation}
\end{proposition}

\begin{proof}
By Lemma \ref{busemann-cocycle},
\[
        \mathcal B_{A_*\mu}(Az)
        =\int\beta_{A\xi}(Az)\,d\mu(\xi)
        =\int\beta_\xi(z)\,d\mu(\xi)+\int c(A,\xi)\,d\mu(\xi).
\]
The second term is independent of \(z\).  Therefore minimizers are carried to
minimizers by \(A\).
\end{proof}

\begin{lemma}[Hessian of a Bergman Busemann function]
Let
\[
        \beta_\eta(w)
        =
        \log
        \frac{|1-\langle w,\eta\rangle|^2}{1-|w|^2},
        \qquad w\in\B^n,\quad \eta\in S^{2n-1}.
\]
Then \(\beta_\eta\) has nonnegative Hessian with respect to the Bergman metric:
\[
        \operatorname{Hess}^{B}_w\beta_\eta(v,v)\ge 0
\]
for every \(w\in\B^n\) and every \(v\in T_w\B^n\). Moreover,
\[
        \operatorname{Hess}^{B}_w\beta_\eta(v,v)=0
\]
if and only if \(v\) is tangent to the Bergman geodesic through \(w\) ending at
\(\eta\).
\end{lemma}

\begin{proof}
This is the standard convexity property of Busemann functions on complex
hyperbolic space. Equivalently, one may verify it directly by using the
homogeneity of the Bergman metric. Indeed, by applying a Bergman isometry, it is
enough to consider the case \(w=0\) and \(\eta=e_1\). In that case
\[
        \beta_{e_1}(z)
        =
        \log |1-z_1|^2-\log(1-|z|^2).
\]
Let \(v=(v_1,\ldots,v_n)\in T_0\B^n\simeq\C^n\), where
\(v_1=a+ib\). Then, for real \(t\) sufficiently small,
\[
\begin{aligned}
        \beta_{e_1}(tv)
        &=
        \log |1-tv_1|^2-\log(1-t^2|v|^2)  \\
        &=
        -2ta
        +
        t^2\left(2b^2+\sum_{j=2}^n |v_j|^2\right)
        +
        O(t^3).
\end{aligned}
\]
Hence
\[
        \frac{d^2}{dt^2}\bigg|_{t=0}\beta_{e_1}(tv)
        =
        2\left(2b^2+\sum_{j=2}^n |v_j|^2\right)\ge 0.
\]
The right-hand side vanishes exactly when
\[
        b=0,\qquad v_2=\cdots=v_n=0,
\]
that is, when \(v\) is a real multiple of \(e_1\). This is precisely the
geodesic direction at \(0\) pointing toward \(e_1\). It remains to explain the reduction to this normal form.  Let
\(w\in\B^n\), \(\eta\in S^{2n-1}\), and choose a Bergman isometry
\(A\in\Aut(\B^n)\) with \(A(w)=0\) and \(A\eta=e_1\).  By the Busemann
cocycle formula,
\[
        \beta_\eta = \beta_{e_1}\circ A + C
\]
near \(w\), for a constant \(C\) independent of the point.  Since \(A\) is an
isometry, it preserves the Levi--Civita connection, and therefore
\[
        \Hess^B_w\beta_\eta(v,v)
        =
        \Hess^B_0\beta_{e_1}(dA_wv,dA_wv).
\]
The additive constant has zero Hessian.  The preceding computation at
\((0,e_1)\) therefore gives the desired inequality and the equality case at
\((w,\eta)\).
\end{proof}

\subsection{Visual measures and the barycentric extension}

Let \(\sigma\) be normalized surface measure on \(S^{2n-1}\).  For
\(z\in\B^n\), define
\begin{equation}\label{visual-measure}
        d\sigma_z(\xi)=P(z,\xi)\,d\sigma(\xi),
        \qquad
        P(z,\xi)=
        \left(\frac{1-|z|^2}{|1-\langle z,\xi\rangle|^2}\right)^n.
\end{equation}
Then \(\sigma_z\) is a probability measure and \(\sigma_0=\sigma\).

\begin{proposition}[Covariance of visual measures]\label{visual-covariance}
For every \(A\in\Aut(\B^n)\) and every \(z\in\B^n\),
\begin{equation}\label{visual-covariance-equation}
        A_*\sigma_z=\sigma_{A(z)}.
\end{equation}
\end{proposition}

\begin{proof}
For holomorphic automorphisms this is the usual transformation law of the
Poisson--Szeg\H{o} kernel \cite{Rudin}. Intrinsically, \(\sigma_z\) is the
visual measure based at \(z\).
\end{proof}

\begin{proposition}[Barycenter of the visual measure]\label{visual-barycenter}
For every \(z\in\B^n\),
\begin{equation}\label{bar-visual}
        \barB(\sigma_z)=z.
\end{equation}
\end{proposition}

\begin{proof}
For \(z=0\), \(\sigma_0=\sigma\) is invariant under \(U(n)\).  Its barycenter
is therefore fixed by \(U(n)\), hence is \(0\).  If \(z\ne0\), choose
\(A\in\Aut(\B^n)\) with \(A(0)=z\).  Then \(\sigma_z=A_*\sigma_0\), and
barycenter equivariance gives
\[
        \barB(\sigma_z)=\barB(A_*\sigma_0)=A(\barB(\sigma_0))=A(0)=z.
\]
\end{proof}

\begin{definition}[Bergman barycentric extension]\label{extension-definition}
Let \(f:S^{2n-1}\to S^{2n-1}\) be a homeomorphism.  Its Bergman barycentric
extension is
\[
        E_B(f):\B^n\to\B^n,
        \qquad
        E_B(f)(z)=\barB(f_*\sigma_z).
\]
\end{definition}

Since \(\sigma_z\) is non-atomic and has full support, \(f_*\sigma_z\) is again
non-atomic and has full support.  Thus the definition is meaningful.

The point \(w=E_B(f)(z)\) is equivalently characterized by
\begin{equation}\label{extension-euler}
        \int_{S^{2n-1}}\nabla_w^B\beta_{f(\xi)}(w)\,d\sigma_z(\xi)=0.
\end{equation}
In coordinates,
\begin{equation}\label{extension-coordinate-euler}
        \int_{S^{2n-1}}
        \frac{f(\xi)}{1-\langle f(\xi),w\rangle}\,d\sigma_z(\xi)
        =
        \frac{w}{1-|w|^2}.
\end{equation}

\subsection{CR boundary geometry and distortion classes}
\label{section-small-distortion}

The boundary sphere carries its CR, equivalently Heisenberg, geometry.  We use
the visual metric
\begin{equation}\label{cr-visual-metric}
        d_{CR}(\xi,\eta)=|1-\langle \xi,\eta\rangle|^{1/2}.
\end{equation}
After one boundary point is removed and the Cayley transform is applied, this
metric is comparable to a Koranyi--Cygan metric on the Heisenberg group.

A distortion function is an increasing homeomorphism
\[
        \eta:[0,\infty)\to[0,\infty).
\]
A homeomorphism \(f:S^{2n-1}\to S^{2n-1}\) is called
\(\eta\)-CR-quasisymmetric if
\[
        \frac{d_{CR}(f(\xi),f(\eta_1))}
             {d_{CR}(f(\xi),f(\eta_2))}
        \le
        \eta\left(
        \frac{d_{CR}(\xi,\eta_1)}
             {d_{CR}(\xi,\eta_2)}
        \right)
\]
for all distinct \(\xi,\eta_1,\eta_2\in S^{2n-1}\).  We call \(f\)
CR-quasisymmetric if it is \(\eta\)-CR-quasisymmetric for some distortion
function \(\eta\).

For four distinct boundary points define the CR metric cross-ratio
\begin{equation}\label{cr-cross-ratio}
        [\xi,\eta,\alpha,\beta]_{CR}
        =
        \frac{d_{CR}(\xi,\alpha)d_{CR}(\eta,\beta)}
        {d_{CR}(\xi,\beta)d_{CR}(\eta,\alpha)}.
\end{equation}
This cross-ratio is invariant under \(\Aut(\B^n)\), because the conformal
factors of a boundary automorphism cancel in \eqref{cr-cross-ratio}.

A homeomorphism \(f:S^{2n-1}\to S^{2n-1}\) is called
\(\theta\)-CR-quasi-M\"obius if
\begin{equation}\label{general-cr-quasimobius}
        [f(\xi),f(\eta),f(\alpha),f(\beta)]_{CR}
        \le
        \theta\!\left([\xi,\eta,\alpha,\beta]_{CR}\right)
\end{equation}
for every quadruple of distinct points, where \(\theta\) is a distortion
function.  On the compact CR sphere, quasisymmetric and quasi-M\"obius control
are quantitatively equivalent \cite{Vaisala,Haissinsky}.  Also, the inverse of
a \(\theta\)-quasi-M\"obius homeomorphism is quasi-M\"obius with distortion
depending only on \(\theta\).

The advantage of quasi-M\"obius control is that it is preserved, with the same
distortion function, by pre- and post-composition with boundary values of ball
automorphisms.

We shall also use the following standard terminology from coarse geometry. A map
\[
F:(X,d_X)\to(Y,d_Y)
\]
is an \((L,C)\)-quasi-isometric embedding if, for all \(x,x'\in X\),
\[
L^{-1}d_X(x,x')-C
\le d_Y(F(x),F(x'))
\le Ld_X(x,x')+C.
\]
It is a quasi-isometry if, in addition, its image is coarsely dense in \(Y\); that is, there exists \(R<\infty\) such that every point of \(Y\) lies within distance \(R\) of \(F(X)\).

A map \(G:Y\to X\) is called a \emph{coarse inverse} of \(F\) if there exists \(D<\infty\) such that
\[
d_X(G(F(x)),x)\le D
\]
for all \(x\in X\), and
\[
d_Y(F(G(y)),y)\le D
\]
for all \(y\in Y\).
We now strengthen the coarse conclusion under small automorphism-invariant
boundary distortion.  The positive CR-M\"obius group is
\[
        \operatorname{Mob}_{CR}^+(S^{2n-1})
        =\{A|_{S^{2n-1}}:A\in\Aut(\B^n)\}.
\]
The metric cross-ratio is also preserved by the anti-CR component
\(C\circ\operatorname{Mob}_{CR}^+(S^{2n-1})\), where
\(C(z)=\overline z\).  In the small-distortion theorem we restrict to the
positive component and call such maps CR-orientation-preserving.  The only
property used below is that a uniform limit of positive maps with cross-ratio
distortion tending to one belongs to the positive, rather than the anti-CR,
component.

\begin{definition}[Almost CR-M\"obius maps]\label{defep}
Let \(\varepsilon\ge0\).  A CR-orientation-preserving homeomorphism
\(f:S^{2n-1}\to S^{2n-1}\) is called \(\varepsilon\)-almost CR-M\"obius if
\begin{equation}\label{almost-cr-mobius}
        e^{-\varepsilon}[\xi,\eta,\alpha,\beta]_{CR}
        \le
        [f(\xi),f(\eta),f(\alpha),f(\beta)]_{CR}
        \le
        e^{\varepsilon}[\xi,\eta,\alpha,\beta]_{CR}
\end{equation}
for every quadruple of distinct boundary points.
\end{definition}

\begin{remark}[Non-emptiness of the class]\label{remark-nonempty-almost-cr-mobius}
Every holomorphic CR-M\"obius transformation is \(0\)-almost CR-M\"obius.
More generally, if \(f\) is CR-orientation-preserving and \(L\)-bi-Lipschitz
with respect to \(d_{CR}\), then
\[
        L^{-4}[\xi,\eta,\alpha,\beta]_{CR}
        \le
        [f(\xi),f(\eta),f(\alpha),f(\beta)]_{CR}
        \le
        L^4[\xi,\eta,\alpha,\beta]_{CR},
\]
so \(f\) is \(\varepsilon\)-almost CR-M\"obius with
\(\varepsilon=4\log L\).  In particular, small nontrivial contact bi-Lipschitz
perturbations of the identity give non-M\"obius examples with arbitrarily small
\(\varepsilon>0\).
\end{remark}

The condition is deliberately automorphism-invariant: if
\(A,B\in\Aut(\B^n)\), then \(A\circ f\circ B\) has the same
\(\varepsilon\)-almost CR-M\"obius distortion.  This invariance is the reason
for using cross-ratio control rather than a fixed quasisymmetric distortion
function in the perturbative theorem.

\subsection{First variation and automorphic normalization}
\label{section-first-variation}

We now compute the normalized differential.  This is the finite-dimensional
formula behind the later obstruction.

\begin{proposition}[First-harmonic formula for the derivative]
\label{first-harmonic-formula}
Let \(f:S^{2n-1}\to S^{2n-1}\) be a homeomorphism and set
\(F=E_B(f)\). Suppose that
\[
        F(0)=0.
\]
Then
\begin{equation}\label{balanced-condition}
        \int_{S^{2n-1}} f(\xi)\,d\sigma(\xi)=0.
\end{equation}
Define real-linear operators \(\mathcal A_f,\mathcal B_f:\C^n\to\C^n\) by
\begin{equation}\label{Af-definition}
        \mathcal A_f u
        =
        u-
        \int_{S^{2n-1}}
        f(\xi)\langle f(\xi),u\rangle\,d\sigma(\xi),
\end{equation}
and
\begin{equation}\label{Bf-definition}
        \mathcal B_f v
        =
        2n\int_{S^{2n-1}}
        f(\xi)\Real\langle v,\xi\rangle\,d\sigma(\xi).
\end{equation}
Then \(F\) is differentiable at \(0\) and
\begin{equation}\label{derivative-formula}
        dF_0=\mathcal A_f^{-1}\mathcal B_f.
\end{equation}
In particular, \(dF_0\) is singular if and only if \(\mathcal B_f\) is
singular.
\end{proposition}

\begin{proof}
The coordinate barycenter equation for \(w=F(z)\) is
\[
        \Psi(z,w)=0,
\]
where
\[
        \Psi(z,w)=
        \int_{S^{2n-1}}
        \frac{f(\xi)}{1-\langle f(\xi),w\rangle}
        P(z,\xi)\,d\sigma(\xi)
        -
        \frac{w}{1-|w|^2}.
\]
Here \(\Psi\) is real-analytic in the variables \((z,w)\) for
\((z,w)\in\B^n\times\B^n\). Indeed, the denominator
\(1-\langle f(\xi),w\rangle\) is locally uniformly bounded away from zero,
and the Poisson--Szeg\H{o} kernel is real-analytic in \(z\). Thus
differentiation under the integral sign is justified. Notice that no
differentiability of \(f\) is required.

Since \(F(0)=0\), the equation \(\Psi(0,0)=0\) gives
\[
        \int_{S^{2n-1}} f(\xi)\,d\sigma(\xi)=0,
\]
which is \eqref{balanced-condition}.

Differentiating \(\Psi(z,F(z))=0\) at \(z=0\), we use
\[
        D_zP(0,\xi)[v]=2n\Real\langle v,\xi\rangle.
\]
Therefore
\[
        D_z\Psi(0,0)[v]=\mathcal B_fv.
\]
On the other hand,
\[
        D_w\Psi(0,0)[u]
        =
        \int_{S^{2n-1}}
        f(\xi)\langle f(\xi),u\rangle\,d\sigma(\xi)-u
        =
        -\mathcal A_fu.
\]
The derivative \(D_w\Psi(0,0)\) is invertible, because it is the Hessian,
up to the standard identification of tangent and cotangent spaces, of the
averaged Busemann potential at its unique non-degenerate minimum. Hence the
implicit function theorem applies, and \(F\) is real-analytic near \(0\).

Now differentiating \(\Psi(z,F(z))=0\) gives
\[
        0
        =
        D_z\Psi(0,0)[v]
        +
        D_w\Psi(0,0)[dF_0v]
        =
        \mathcal B_fv-\mathcal A_f(dF_0v).
\]
Thus
\[
        dF_0v=\mathcal A_f^{-1}\mathcal B_fv,
\]
which proves \eqref{derivative-formula}. Since \(\mathcal A_f\) is invertible,
\(dF_0\) is singular if and only if \(\mathcal B_f\) is singular.
\end{proof}

\begin{remark}[The identity case]
For \(f=\Id\), unitary symmetry gives
\[
        \int_{S^{2n-1}}\xi\langle \xi,u\rangle\,d\sigma(\xi)=0
\]
and
\[
        2n\int_{S^{2n-1}}\xi\,\Real\langle v,\xi\rangle\,d\sigma(\xi)=v.
\]
Thus \(\mathcal A_{\Id}=I\), \(\mathcal B_{\Id}=I\), and
\(dE_B(\Id)_0=I\).
\end{remark}

\subsubsection{Automorphic normalization}
\label{section-automorphic-estimates}

For a real-linear map \(L\) between Euclidean spaces, we write
\[
        \|L\|=\sup_{|v|=1}|Lv|,
        \qquad
        \ell(L)=\inf_{|v|=1}|Lv|
\]
for its maximal and minimal singular values.  If \(F:\B^n\to\B^n\) is
differentiable, then
\[
        \|dF_z\|_{g_B},
        \qquad
        \ell_{g_B}(dF_z)
\]
denote the corresponding maximal and minimal singular values computed with
respect to the Bergman metric.  We set
\[
        K_B(F,z)=
        \frac{\|dF_z\|_{g_B}}{\ell_{g_B}(dF_z)}
\]
whenever \(\ell_{g_B}(dF_z)>0\), and
\[
        K_B(F)=\sup_{z\in\B^n}K_B(F,z).
\]

Let \(f:S^{2n-1}\to S^{2n-1}\) be a boundary homeomorphism and set
\[
        F=E_B(f).
\]
Fix \(z_0\in\B^n\).  Choose automorphisms \(B,A\in\Aut(\B^n)\) such that
\[
        B(0)=z_0,
        \qquad
        A(F(z_0))=0.
\]
Define
\[
        g=A\circ f\circ B|_{S^{2n-1}}.
\]
Then naturality gives
\[
        E_B(g)=A\circ F\circ B,
        \qquad
        E_B(g)(0)=0.
\]
Moreover,
\[
        dE_B(g)_0
        =
        dA_{F(z_0)}
        \circ
        dF_{z_0}
        \circ
        dB_0.
\]
Since \(A\) and \(B\) are Bergman isometries, singular values measured with
respect to \(g_B\) are unchanged by these pre- and post-compositions.  Hence
\[
        K_B(F,z_0)=K_B(E_B(g),0).
\]
For normalized \(g\), Proposition~\ref{first-harmonic-formula} gives
\[
        dE_B(g)_0=\mathcal A_g^{-1}\mathcal B_g.
\]
Thus every local distortion estimate for \(E_B(f)\) reduces, after
automorphic normalization, to estimating the real-linear operator
\[
        \mathcal A_g^{-1}\mathcal B_g
\]
at the origin.

\section{Proofs of the main results}
\label{section-main-proofs}

\subsection{Basic properties of the barycentric extension}

\begin{theorem}[Well-definedness and naturality]\label{naturality-theorem}
For every boundary homeomorphism \(f:S^{2n-1}\to S^{2n-1}\), the map
\(E_B(f)\) is well-defined.  If \(A,B\in\Aut(\B^n)\), then
\begin{equation}\label{naturality-equation}
        E_B(A\circ f\circ B)=A\circ E_B(f)\circ B.
\end{equation}
In particular, if \(f=A|_{S^{2n-1}}\), then \(E_B(f)=A\).
\end{theorem}

\begin{proof}
Well-definedness follows from Proposition \ref{existence-uniqueness-barycenter}.
For naturality, using the covariance of visual measures and barycenter
equivariance,
\[
\begin{aligned}
        E_B(A\circ f\circ B)(z)
        &=\barB((A\circ f\circ B)_*\sigma_z)\\
        &=\barB(A_*f_*\sigma_{B(z)})\\
        &=A\bigl(\barB(f_*\sigma_{B(z)})\bigr)\\
        &=A(E_B(f)(B(z))).
\end{aligned}
\]
If \(f=A|_{S^{2n-1}}\), then
\[
        E_B(f)(z)=\barB(A_*\sigma_z)=A(\barB(\sigma_z))=A(z).
\]
\end{proof}

\begin{theorem}[Continuity in the ball]\label{interior-continuity}
Let \(f:S^{2n-1}\to S^{2n-1}\) be a homeomorphism.  Then \(E_B(f)\) is
continuous on \(\B^n\).
\end{theorem}

\begin{proof}
If \(z_j\to z\) in \(\B^n\), then the densities \(P(z_j,\xi)\) converge
uniformly to \(P(z,\xi)\) on \(S^{2n-1}\).  Hence \(\sigma_{z_j}\to\sigma_z\)
weakly, and therefore \(f_*\sigma_{z_j}\to f_*\sigma_z\) weakly.

Let \(w_j=E_B(f)(z_j)\).  The properness estimate in Proposition
\ref{existence-uniqueness-barycenter} prevents \(w_j\) from escaping to the
boundary along a subsequence.  Thus every subsequential limit lies in \(\B^n\)
and minimizes the limiting potential \(\mathcal B_{f_*\sigma_z}\).  By
uniqueness of the minimizer, this limit is \(E_B(f)(z)\).  Hence
\(w_j\to E_B(f)(z)\).
\end{proof}

\begin{theorem}[Boundary extension]\label{boundary-extension}
Let \(f:S^{2n-1}\to S^{2n-1}\) be a homeomorphism. Then
\[
        \lim_{z\to\eta}E_B(f)(z)=f(\eta),
        \qquad \eta\in S^{2n-1}.
\]
Consequently, \(E_B(f)\) extends continuously to
\(\overline{\B^n}\) with boundary value \(f\).
\end{theorem}

\begin{proof}
Fix \(\eta\in S^{2n-1}\) and put
\[
        p=f(\eta).
\]
Let \(z_j\to\eta\) in \(\B^n\), and set
\[
        \mu_j=f_*\sigma_{z_j},
        \qquad
        w_j=E_B(f)(z_j)=\barB(\mu_j).
\]
Since the Poisson--Szeg\H{o} measures form an approximate identity on
\(S^{2n-1}\),
\[
        \sigma_{z_j}\rightharpoonup\delta_\eta.
\]
As \(f\) is continuous, it follows that
\[
        \mu_j=f_*\sigma_{z_j}
        \rightharpoonup
        \delta_p.
\]
Indeed, for every continuous function
\(\varphi:S^{2n-1}\to\mathbb R\),
\[
\int_{S^{2n-1}}\varphi(\xi)\,d\mu_j(\xi)
=
\int_{S^{2n-1}}\varphi(f(\xi))\,d\sigma_{z_j}(\xi)
\longrightarrow
\varphi(f(\eta))
=
\varphi(p).
\]

We claim that \(w_j\to p\). Since the closed Euclidean ball
\(\overline{\B^n}\) is compact, it is enough to show that every cluster point
of \((w_j)\) is \(p\). Passing to a subsequence, suppose that
\[
        w_j\longrightarrow w_\infty\in\overline{\B^n}.
\]

First suppose that \(w_\infty\in\B^n\). For a probability measure \(\mu\) on
\(S^{2n-1}\), recall that
\[
        \mathcal B_\mu(w)
        =
        \int_{S^{2n-1}}\beta_\xi(w)\,d\mu(\xi).
\]
On every compact subset \(K\Subset\B^n\), the function
\[
        (\xi,w)\longmapsto\beta_\xi(w)
\]
is continuous on the compact set \(S^{2n-1}\times K\). Hence the weak
convergence \(\mu_j\rightharpoonup\delta_p\) implies
\[
        \mathcal B_{\mu_j}
        \longrightarrow
        \beta_p
\]
uniformly on \(K\).

Since \(w_j\) minimizes \(\mathcal B_{\mu_j}\), for every fixed
\(u\in\B^n\),
\[
        \mathcal B_{\mu_j}(w_j)
        \le
        \mathcal B_{\mu_j}(u).
\]
Choosing a compact subset of \(\B^n\) containing \(u\), \(w_\infty\), and
\(w_j\) for all sufficiently large \(j\), and passing to the limit, we obtain
\[
        \beta_p(w_\infty)\le\beta_p(u).
\]
Since \(u\in\B^n\) was arbitrary, \(w_\infty\) would be a global minimum of
\(\beta_p\) in \(\B^n\). This is impossible, because along the radial
geodesic \(r\mapsto rp\),
\[
        \beta_p(rp)
        =
        \log\frac{|1-r|^2}{1-r^2}
        =
        \log\frac{1-r}{1+r}
        \longrightarrow -\infty
        \qquad (r\to1^-).
\]
Thus no cluster point of \((w_j)\) can lie in \(\B^n\).

It remains to exclude a boundary cluster point different from \(p\).
Suppose therefore that
\[
        w_j\longrightarrow q\in S^{2n-1},
        \qquad q\ne p.
\]
Since \(w_j=\barB(\mu_j)\), the coordinate barycenter equation gives
\[
        \int_{S^{2n-1}}
        \frac{\xi}{1-\langle \xi,w_j\rangle}\,d\mu_j(\xi)
        =
        \frac{w_j}{1-|w_j|^2}.
\]
Multiplying by \(1-|w_j|^2\), we obtain
\begin{equation}\label{boundary-barycenter-scaled}
        \int_{S^{2n-1}}
        \frac{(1-|w_j|^2)\xi}
             {1-\langle \xi,w_j\rangle}\,d\mu_j(\xi)
        =
        w_j.
\end{equation}

Because \(p\ne q\),
\[
        |1-\langle p,q\rangle|>0.
\]
Hence we may choose a neighborhood \(U\) of \(p\) and a constant \(c>0\)
such that, for all sufficiently large \(j\),
\[
        |1-\langle \xi,w_j\rangle|\ge c,
        \qquad \xi\in U.
\]
Therefore
\[
\sup_{\xi\in U}
\left|
\frac{(1-|w_j|^2)\xi}
     {1-\langle \xi,w_j\rangle}
\right|
\le
\frac{1-|w_j|^2}{c}
\longrightarrow0.
\]

On the other hand, for every \(\xi\in S^{2n-1}\),
\[
        |1-\langle \xi,w_j\rangle|
        \ge 1-|w_j|,
\]
and hence
\[
\left|
\frac{(1-|w_j|^2)\xi}
     {1-\langle \xi,w_j\rangle}
\right|
\le
\frac{1-|w_j|^2}{1-|w_j|}
=
1+|w_j|
\le2.
\]
Since \(\mu_j\rightharpoonup\delta_p\) and \(U\) is a neighborhood of \(p\),
\[
        \mu_j(S^{2n-1}\setminus U)\longrightarrow0.
\]
Consequently,
\[
\begin{aligned}
\left|
\int_{S^{2n-1}}
\frac{(1-|w_j|^2)\xi}
     {1-\langle \xi,w_j\rangle}\,d\mu_j(\xi)
\right|
&\le
\sup_{\xi\in U}
\left|
\frac{(1-|w_j|^2)\xi}
     {1-\langle \xi,w_j\rangle}
\right|  \\
&\qquad
+2\,\mu_j(S^{2n-1}\setminus U)
\longrightarrow0.
\end{aligned}
\]
But by \eqref{boundary-barycenter-scaled} the left-hand side equals
\(|w_j|\), which tends to \(1\). This is a contradiction.

We have therefore ruled out every cluster point except \(p\). Hence
\[
        w_j=E_B(f)(z_j)\longrightarrow p=f(\eta).
\]
Since the sequence \(z_j\to\eta\) was arbitrary,
\[
        \lim_{z\to\eta}E_B(f)(z)=f(\eta).
\]
Together with the continuity of \(E_B(f)\) in \(\B^n\), this gives a
continuous extension to \(\overline{\B^n}\) whose boundary value is \(f\).
\end{proof}
\begin{theorem}[Interior real-analyticity]\label{smooth-regularity}
Let \(f:S^{2n-1}\to S^{2n-1}\) be a homeomorphism. Then its Bergman
barycentric extension
\[
        E_B(f):\B^n\to\B^n
\]
is real-analytic in the interior of \(\B^n\). More generally, the same
conclusion holds for every continuous map \(f:S^{2n-1}\to S^{2n-1}\) such that,
for every \(z\in\B^n\), the measure \(f_*\sigma_z\) is admissible and has full
support on \(S^{2n-1}\).
\end{theorem}

\begin{proof}
For \(z,w\in\B^n\), define
\[
        \Phi_f(z,w)
        =
        \int_{S^{2n-1}}
        \frac{f(\xi)}{1-\langle f(\xi),w\rangle}
        P(z,\xi)\,\dd\sigma(\xi)
        -
        \frac{w}{1-|w|^2},
\]
where
\[
        P(z,\xi)
        =
        \left(
        \frac{1-|z|^2}{|1-\langle z,\xi\rangle|^2}
        \right)^n
\]
is the Poisson--Szeg\H{o} kernel. The barycenter equation for
\(w=E_B(f)(z)\) is
\[
        \Phi_f(z,w)=0.
\]

We first prove that \(\Phi_f\) is real-analytic on
\(\B^n\times\B^n\). Fix \((z_0,w_0)\in\B^n\times\B^n\). Choose relatively
compact neighborhoods \(U,V\Subset\B^n\) of \(z_0,w_0\), respectively. Then
there are constants \(r,\rho<1\) such that
\[
        |z|\le r,\qquad |w|\le \rho
\]
for all \(z\in U\), \(w\in V\). Hence, uniformly in
\(\xi\in S^{2n-1}\),
\[
        |1-\langle f(\xi),w\rangle|
        \ge 1-\rho>0
\]
and
\[
        |1-\langle z,\xi\rangle|
        \ge 1-r>0.
\]
Thus the denominators occurring in the integrand are bounded away from zero
on \(U\times V\), uniformly in \(\xi\).

Since \(f(\xi)\in S^{2n-1}\), all derivatives of
\[
        (z,w)\mapsto
        \frac{f(\xi)}{1-\langle f(\xi),w\rangle}P(z,\xi)
\]
with respect to the real and imaginary parts of \(z\) and \(w\) are locally
uniformly bounded by constants depending only on \(U,V\) and the order of
differentiation. Differentiation under the integral is therefore justified to
all orders. The integrand is real-analytic in \((z,w)\), and hence
\(\Phi_f\) is real-analytic on \(\B^n\times\B^n\).

Now fix \(z_0\in\B^n\) and put
\[
        w_0=E_B(f)(z_0).
\]
Let
\[
        \mu_0=f_*\sigma_{z_0}.
\]
The point \(w_0\) is the unique minimizer of the averaged Busemann potential
\[
        \mathcal B_{\mu_0}(y)
        =
        \int_{S^{2n-1}}\beta_\eta(y)\,\dd\mu_0(\eta),
        \qquad y\in\B^n.
\]
The equation
\[
        \Phi_f(z_0,w_0)=0
\]
is precisely the vanishing of the real gradient of
\(\mathcal B_{\mu_0}\) at \(w_0\), written in complex coordinates. Therefore
\(D_w\Phi_f(z_0,w_0)\) is, up to the standard identification of complex
coordinates with real tangent coordinates, the Hessian of
\(\mathcal B_{\mu_0}\) at \(w_0\).

We claim that this averaged Hessian is positive definite. Indeed, for every
\(\eta\in S^{2n-1}\), the Busemann function
\[
        y\mapsto \beta_\eta(y)
\]
has nonnegative Hessian in the Bergman metric:
\[
        \Hess_y\beta_\eta(v,v)\ge 0
\]
for every \(y\in\B^n\) and every \(v\in T_y\B^n\). Moreover, for fixed
\(y\in\B^n\) and fixed nonzero \(v\in T_y\B^n\), the equality
\[
        \Hess_y\beta_\eta(v,v)=0
\]
can occur only when \(v\) is tangent to the geodesic through \(y\) having
endpoint \(\eta\) at infinity. Thus the set
\[
        Z_{y,v}
        :=
        \left\{
        \eta\in S^{2n-1}:
        \Hess_y\beta_\eta(v,v)=0
        \right\}
\]
is a proper closed subset of \(S^{2n-1}\).

Since \(\mu_0=f_*\sigma_{z_0}\) has full support on \(S^{2n-1}\), this proper
closed subset cannot have full \(\mu_0\)-measure. Consequently, for every
nonzero \(v\in T_{w_0}\B^n\),
\[
\begin{aligned}
        \Hess_{w_0}\mathcal B_{\mu_0}(v,v)
        &=
        \int_{S^{2n-1}}
        \Hess_{w_0}\beta_\eta(v,v)\,\dd\mu_0(\eta)  \\
        &>0.
\end{aligned}
\]
Hence
\[
        \Hess_{w_0}\mathcal B_{\mu_0}
\]
is positive definite. Therefore \(D_w\Phi_f(z_0,w_0)\) is invertible.

By the real-analytic implicit function theorem, there exist neighborhoods
\(U_0\subset\B^n\) of \(z_0\) and \(V_0\subset\B^n\) of \(w_0\), and a
real-analytic map
\[
        W:U_0\to V_0
\]
such that
\[
        \Phi_f(z,W(z))=0
\]
for all \(z\in U_0\). By uniqueness of the barycenter, one has
\[
        W(z)=E_B(f)(z)
\]
for \(z\in U_0\). Thus \(E_B(f)\) is real-analytic in a neighborhood of
\(z_0\). Since \(z_0\in\B^n\) was arbitrary, \(E_B(f)\) is real-analytic
throughout \(\B^n\).
\end{proof}

\subsection{Proof of the quasi-isometry theorem}
\label{section-general-quasi-isometry}

\begin{lemma}[Compactness of the barycentrically normalized class]
\label{lemma-normalized-qm-compactness}
Fix a distortion function \(\theta\), and let
\[
        \mathcal Q_\theta^0
        =\{g:S^{2n-1}\to S^{2n-1}:
        g\text{ is }\theta\text{-CR-quasi-M\"obius and }E_B(g)(0)=0\}.
\]
Then \(\mathcal Q_\theta^0\) is compact in the uniform topology.  Moreover,
\[
        (\mathcal Q_\theta^0)^{-1}
        =\{g^{-1}:g\in\mathcal Q_\theta^0\}
\]
is compact in the uniform topology.
\end{lemma}

\begin{proof}
Let \(g_j\in\mathcal Q_\theta^0\).  The barycenter equation at the origin gives
\begin{equation}\label{normalized-balance-qm}
        \int_{S^{2n-1}}g_j(\xi)\,d\sigma(\xi)=0.
\end{equation}
We use the standard compactness alternative for uniformly quasi-M\"obius maps
of compact metric spaces; see \cite[Corollary~2.7]{Haissinsky}.  After passing
to a subsequence, either the maps form a normalized equicontinuous family, or
there are \(a,b\in S^{2n-1}\) such that
\[
        g_j\longrightarrow b
\]
uniformly on compact subsets of \(S^{2n-1}\setminus\{a\}\).
The second alternative is impossible: since \(\sigma\) is non-atomic,
dominated convergence in \eqref{normalized-balance-qm} would give
\[
        0=\lim_j\int g_j\,d\sigma=b,
\]
contrary to \(|b|=1\).

Hence a subsequence converges uniformly to a nonconstant map \(g\).  Standard
normalized quasi-M\"obius compactness gives that the limit is a homeomorphism;
alternatively, the two-sided cross-ratio bounds obtained from
\eqref{general-cr-quasimobius} by permuting the four points rule out
identification of two distinct points, while surjectivity follows from compactness
and the surjectivity of the \(g_j\).  The quasi-M\"obius inequality and
\eqref{normalized-balance-qm} pass to the limit, so
\(g\in\mathcal Q_\theta^0\).  Thus \(\mathcal Q_\theta^0\) is compact.

If \(g_j\to g\) uniformly with all maps and the limit homeomorphisms of the
compact sphere, then \(g_j^{-1}\to g^{-1}\) uniformly.  Therefore inversion is
continuous on \(\mathcal Q_\theta^0\), and the inverse family is compact.
\end{proof}

\subsection{Standard topological facts}

We use the following standard facts. First, the classical theorem of turning
tangents asserts that a regular simple closed plane curve has rotation index
\(\pm1\), with the sign determined by its orientation; see
\cite[Sec.~5--7, Theorem~2, p.~396]{doCarmo1976}. Thus, for a positively
oriented smooth regular Jordan curve \(\Gamma\),
\begin{equation}\label{eq:turning-standard}
    \wind(\Gamma',0)=1.
\end{equation}
The same conclusion holds for regular \(C^1\) Jordan curves, for instance by
\(C^1\)-approximation by smooth embeddings and homotopy invariance of the
winding number.

Second, if \(u\in C(\overline{\D},\C)\) and \(y\notin u(\T)\), then the planar
Brouwer degree is given by the winding number of the boundary map,
\begin{equation}\label{eq:degree-winding}
    \deg(u,\D,y)=\wind(u|_{\T},y),
\end{equation}
and, when \(y\) is a regular value, the degree is the sum of the local signs
of the Jacobian over the preimages of \(y\); see
\cite[Chapter~1]{Deimling1985}.

\begin{lemma}[Joint continuity]
\label{lemma-joint-continuity-barycentric}
If \(f_j,f:S^{2n-1}\to S^{2n-1}\) are homeomorphisms with
\(f_j\to f\) uniformly and \(z_j\to z\in\B^n\), then
\[
        E_B(f_j)(z_j)\longrightarrow E_B(f)(z).
\]
\end{lemma}

\begin{proof}
Set
\[
        \mu_j=(f_j)_*\sigma_{z_j},
        \qquad
        \mu=f_*\sigma_z.
\]
We first claim that
\[
        \mu_j\rightharpoonup\mu.
\]
Indeed, for every \(\varphi\in C(S^{2n-1})\),
\[
\int_{S^{2n-1}}\varphi\,d\mu_j
=
\int_{S^{2n-1}}
\varphi(f_j(\xi))P(z_j,\xi)\,d\sigma(\xi).
\]
Since \(z_j\to z\in\B^n\), the points \(z_j\) eventually lie in a fixed
compact subset of \(\B^n\), and hence
\[
        P(z_j,\cdot)\longrightarrow P(z,\cdot)
\]
uniformly on \(S^{2n-1}\). Moreover, \(f_j\to f\) uniformly, so the uniform
continuity of \(\varphi\) gives
\[
        \varphi\circ f_j\longrightarrow\varphi\circ f
\]
uniformly. It follows that
\[
        \int\varphi\,d\mu_j\longrightarrow\int\varphi\,d\mu,
\]
and therefore \(\mu_j\rightharpoonup\mu\).

Let
\[
        w_j=\barB(\mu_j)=E_B(f_j)(z_j).
\]
We claim that \((w_j)\) stays in a compact subset of \(\B^n\). Suppose
otherwise. Passing to a subsequence, we may write
\[
        w_j=r_j\zeta_j,
        \qquad
        r_j\to1,
        \qquad
        \zeta_j\to\zeta\in S^{2n-1}.
\]
The measure \(\mu=f_*\sigma_z\) is non-atomic. Hence
\(\mu(\{\zeta\})=0\), and we may choose a closed neighborhood \(K\) of
\(\zeta\) and a number \(a<1/2\) such that
\[
        \mu(K)<a.
\]
Since \(K\) is closed, the Portmanteau theorem gives
\[
        \limsup_{j\to\infty}\mu_j(K)\le \mu(K)<a,
\]
so, for all sufficiently large \(j\),
\[
        m_j:=\mu_j(K)\le a.
\]

Recall that
\[
        \beta_\xi(w)
        =
        \log\frac{|1-\langle w,\xi\rangle|^2}{1-|w|^2}.
\]
For every \(\xi\in S^{2n-1}\),
\[
        |1-r_j\langle\zeta_j,\xi\rangle|
        \ge 1-r_j,
\]
and therefore
\begin{equation}\label{eq:busemann-global-lower}
        \beta_\xi(w_j)
        \ge
        \log\frac{1-r_j}{1+r_j}.
\end{equation}

On the other hand, since \(K\) is a neighborhood of \(\zeta\), the compact
set \(S^{2n-1}\setminus K\) does not contain \(\zeta\). Thus
\[
        c_0:=
        \min_{\xi\in S^{2n-1}\setminus K}
        |1-\langle\zeta,\xi\rangle|>0.
\]
Because \(r_j\zeta_j\to\zeta\), for all sufficiently large \(j\),
\[
        |1-r_j\langle\zeta_j,\xi\rangle|
        \ge \frac{c_0}{2},
        \qquad
        \xi\in S^{2n-1}\setminus K.
\]
Consequently,
\begin{equation}\label{eq:busemann-away-lower}
        \beta_\xi(w_j)
        \ge
        -\log(1-r_j^2)-C,
        \qquad
        \xi\in S^{2n-1}\setminus K,
\end{equation}
for some constant \(C\) independent of \(j\).

Using \eqref{eq:busemann-global-lower} on \(K\) and
\eqref{eq:busemann-away-lower} on its complement, we obtain
\[
\begin{aligned}
        \mathcal B_{\mu_j}(w_j)
        &\ge
        m_j\log\frac{1-r_j}{1+r_j}
        +(1-m_j)\bigl[-\log(1-r_j^2)-C\bigr] \\
        &=
        (1-2m_j)|\log(1-r_j)|
        -\log(1+r_j)-(1-m_j)C.
\end{aligned}
\]
Since \(m_j\le a<1/2\) and \(1\le 1+r_j\le2\), it follows that
\[
        \mathcal B_{\mu_j}(w_j)
        \ge
        (1-2a)|\log(1-r_j)|-C'
        \longrightarrow+\infty
\]
for some constant \(C'\). This contradicts the minimizing property of
\(w_j\), since
\[
        \mathcal B_{\mu_j}(w_j)
        \le
        \mathcal B_{\mu_j}(0)
        =0.
\]
Hence \((w_j)\) is relatively compact in \(\B^n\).

We next show that
\[
        \mathcal B_{\mu_j}\longrightarrow\mathcal B_\mu
\]
locally uniformly on \(\B^n\). Let \(L\Subset\B^n\). Since
\[
        (\xi,w)\longmapsto\beta_\xi(w)
\]
is continuous on the compact set \(S^{2n-1}\times L\), the map
\[
        L\longrightarrow C(S^{2n-1}),
        \qquad
        w\longmapsto\beta_{\cdot}(w),
\]
is continuous with respect to the uniform norm. Thus
\[
        \mathcal K_L
        :=
        \{\beta_{\cdot}(w):w\in L\}
\]
is compact in \(C(S^{2n-1})\).

Define
\[
        T_j(\varphi)
        =
        \int_{S^{2n-1}}\varphi\,d(\mu_j-\mu).
\]
Then \(\|T_j\|\le2\), while \(T_j(\varphi)\to0\) for every
\(\varphi\in C(S^{2n-1})\). We claim that this convergence is uniform on
\(\mathcal K_L\). Indeed, given \(\varepsilon>0\), choose
\(\varphi_1,\ldots,\varphi_N\in\mathcal K_L\) such that every
\(\varphi\in\mathcal K_L\) satisfies
\[
        \|\varphi-\varphi_k\|_\infty<\varepsilon
\]
for some \(k\). Then
\[
        |T_j(\varphi)|
        \le
        |T_j(\varphi_k)|+2\varepsilon.
\]
Taking the supremum over \(\varphi\in\mathcal K_L\) and then letting
\(j\to\infty\) gives
\[
        \limsup_{j\to\infty}
        \sup_{\varphi\in\mathcal K_L}|T_j(\varphi)|
        \le2\varepsilon.
\]
Since \(\varepsilon>0\) is arbitrary,
\[
        \sup_{w\in L}
        |\mathcal B_{\mu_j}(w)-\mathcal B_\mu(w)|
        \longrightarrow0.
\]

Now let \(w_j\to w_\infty\in\B^n\) along a subsequence. For every fixed
\(v\in\B^n\), the minimizing property of \(w_j\) gives
\[
        \mathcal B_{\mu_j}(w_j)
        \le
        \mathcal B_{\mu_j}(v).
\]
Choose \(L\Subset\B^n\) containing \(v\), \(w_\infty\), and \(w_j\) for all
sufficiently large \(j\). By the local uniform convergence just proved,
together with the continuity of \(\mathcal B_\mu\),
\[
        \mathcal B_\mu(w_\infty)
        \le
        \mathcal B_\mu(v).
\]
Since \(v\in\B^n\) was arbitrary, \(w_\infty\) minimizes
\(\mathcal B_\mu\).

Finally, \(\mu=f_*\sigma_z\) is non-atomic and has full support, because
\(\sigma_z\) has these properties and \(f\) is a homeomorphism. Hence
Proposition~\ref{existence-uniqueness-barycenter} implies that its
barycenter is unique. Therefore
\[
        w_\infty=\barB(\mu)=E_B(f)(z).
\]
Thus every cluster point of \((w_j)\) equals \(E_B(f)(z)\). Since
\((w_j)\) is relatively compact in \(\B^n\), it follows that
\[
        E_B(f_j)(z_j)=w_j\longrightarrow E_B(f)(z).
\]
\end{proof}

\begin{proposition}[Uniform control on bounded normalized sets]
\label{proposition-uniform-bounded-displacement}
For every distortion function \(\theta\) and every \(R>0\) there exists
\(C=C(\theta,n,R)<\infty\) such that
\[
        d_B(0,E_B(g)(z))\le C
\]
whenever \(g\in\mathcal Q_\theta^0\) and \(d_B(0,z)\le R\).
\end{proposition}

\begin{proof}
The closed Bergman ball \(K_R=\{z:d_B(0,z)\le R\}\) is compact.  By
Lemma~\ref{lemma-normalized-qm-compactness}, \(\mathcal Q_\theta^0\) is compact,
and by Lemma~\ref{lemma-joint-continuity-barycentric} the map
\[
        (g,z)\longmapsto E_B(g)(z)
\]
is continuous on \(\mathcal Q_\theta^0\times K_R\).  Its image is therefore a
compact subset of \(\B^n\), on which \(d_B(0,\cdot)\) is bounded.
\end{proof}

\begin{proof}[Proof of Theorem~\ref{theorem-general-quasi-isometry}]

The \(\eta\)-quasisymmetry of \(f\) gives a
\(\theta\)-quasi-M\"obius distortion depending only on \(\eta\).  Fix
\(x\in\B^n\), choose \(B_x,A_x\in\Aut(\B^n)\) with
\[
        B_x(0)=x,
        \qquad A_x(F(x))=0,
\]
and set
\[
        g_x=A_x\circ f\circ B_x|_{S^{2n-1}}.
\]
Naturality gives
\[
        E_B(g_x)=A_x\circ F\circ B_x,
        \qquad E_B(g_x)(0)=0.
\]
Since ball automorphisms preserve the CR cross-ratio,
\(g_x\in\mathcal Q_\theta^0\) for every \(x\).

Apply Proposition~\ref{proposition-uniform-bounded-displacement} with \(R=1\).
There is \(C_+<\infty\), depending only on \(n\) and \(\eta\), such that if
\(d_B(x,y)\le1\), then, with \(u=B_x^{-1}(y)\),
\[
\begin{aligned}
        d_B(F(x),F(y))
        &=d_B(0,E_B(g_x)(u))
        \le C_+.
\end{aligned}
\]
Dividing a Bergman geodesic from \(x\) to \(y\) into subsegments of length at
most one yields
\begin{equation}\label{coarse-upper-F}
        d_B(F(x),F(y))\le C_+d_B(x,y)+C_+.
\end{equation}

Put \(G=E_B(f^{-1})\).  Since
\[
        g_x^{-1}=B_x^{-1}\circ f^{-1}\circ A_x^{-1},
\]
naturality gives
\[
        E_B(g_x^{-1})(0)=B_x^{-1}(G(F(x))).
\]
The inverse family \((\mathcal Q_\theta^0)^{-1}\) is compact by
Lemma~\ref{lemma-normalized-qm-compactness}; hence joint continuity gives a
constant \(D_1<\infty\) such that
\begin{equation}\label{GF-close-id}
        d_B(G(F(x)),x)\le D_1
\end{equation}
for every \(x\).  Applying the same argument to \(f^{-1}\) gives
\begin{equation}\label{FG-close-id}
        d_B(F(G(y)),y)\le D_2
\end{equation}
for every \(y\), with \(D_2\) depending only on \(n\) and \(\eta\).  Thus \(F\)
is coarsely onto and \(G\) is a coarse inverse.

Finally, the already proved upper estimate applied to \(G\) gives constants
\(C_-\) depending only on \(n\) and \(\eta\) such that
\[
        d_B(G(p),G(q))\le C_-d_B(p,q)+C_-.
\]
Therefore
\[
\begin{aligned}
        d_B(x,y)
        &\le d_B(x,G(F(x)))
        +d_B(G(F(x)),G(F(y)))
        +d_B(G(F(y)),y)\\
        &\le 2D_1+C_-d_B(F(x),F(y))+C_-.
\end{aligned}
\]
Together with \eqref{coarse-upper-F}, this gives
\eqref{general-quasi-isometry-estimate} after enlarging the constants.  Taking
\(D=\max\{D_1,D_2\}\) gives \eqref{coarse-inverse-estimates}.
\end{proof}

\begin{remark}[Large scale versus local behavior]
\label{remark-qi-versus-local-collapse}
Theorem~\ref{theorem-general-quasi-isometry} is a large-scale statement.  It
does not imply injectivity, local non-degeneracy, or quasiconformality in the
differential sense.  The counterexample in
Section~\ref{section-cr-qs-counterexample} shows that these stronger properties
can fail even for smooth contact boundary data.
\end{remark}

\subsection{Proof of the Tukia-type almost-isometry theorem}

We shall use the same quasi-M\"obius compactness alternative as above
\cite[Corollary~2.7]{Haissinsky}.  In the small-distortion regime it converts
automorphic normalization into uniform closeness to unitary normal forms.

\begin{lemma}[Small-distortion normalization]
\label{lemma-small-distortion-normalization}
Let \(\varepsilon_j\to0\), and let
\(g_j:S^{2n-1}\to S^{2n-1}\) be CR-orientation-preserving
\(\varepsilon_j\)-almost CR-M\"obius homeomorphisms such that
\[
        E_B(g_j)(0)=0.
\]
Then there exist unitary maps \(U_j\in U(n)\) such that
\[
        \|U_j^{-1}\circ g_j-\Id\|_{L^\infty(S^{2n-1})}\to0.
\]
Equivalently, every sufficiently small-distortion automorphically normalized
positive boundary map is uniformly close to a unitary map.
\end{lemma}

\begin{proof}
Suppose that the conclusion fails.  We use the standard compactness alternative
for uniformly quasi-M\"obius self-homeomorphisms of compact metric spaces
\cite[Corollary~2.7]{Haissinsky}: if the quasi-M\"obius
distortion tends to one, then a subsequence either converges uniformly to a
cross-ratio preserving homeomorphism, or collapses uniformly on compact subsets
of the complement of at most one point to a constant boundary point.  The
second alternative is incompatible with the barycentric normalization.  Indeed,
\(E_B(g_j)(0)=0\) implies, by Proposition~\ref{first-harmonic-formula},
\[
        \int_{S^{2n-1}} g_j(\xi)\,d\sigma(\xi)=0.
\]
If a collapsing alternative occurred, then, since \(\sigma\) is non-atomic
and the exceptional point has \(\sigma\)-measure zero, dominated convergence
would force these integrals to converge to a point of \(S^{2n-1}\), a
contradiction.  Hence, after passing to a subsequence, \(g_j\) converges
uniformly to a homeomorphism \(g_\infty\).

Since \(\varepsilon_j\to0\), the limit preserves the CR metric cross-ratio.
The metric cross-ratio alone would allow both the holomorphic and the
anti-holomorphic CR-M\"obius groups.  The CR-orientation-preserving assumption
excludes the anti-CR component.  Hence, by the Liouville theorem for the CR
sphere \cite{Pansu,KoranyiReimann}, there exists \(G\in\Aut(\B^n)\) such that
\[
        g_\infty=G|_{S^{2n-1}}.
\]

Continuity of the barycentric extension under uniform convergence of boundary
maps gives
\[
        E_B(g_\infty)(0)=0.
\]
By automorphism naturality, \(E_B(g_\infty)=G\). Hence \(G(0)=0\), and
therefore \(G\in U(n)\). Thus \(G^{-1}\circ g_j\to\Id\) uniformly,
contradicting the assumed failure of the conclusion.
\end{proof}

\subsubsection{Perturbative control at the normalized origin}
\label{section-perturbative-distortion}

The next estimate is the differential core of the small-distortion theorem.  It
strengthens the preceding non-collapse estimate by controlling the full
distortion of the normalized differential.

\begin{lemma}[Perturbative control of \(\mathcal B_g\)]\label{lemma-perturbative-Bg}
Let \(g:S^{2n-1}\to S^{2n-1}\) satisfy
\[
        \|g-\Id\|_{L^\infty(S^{2n-1})}\le\delta.
\]
Then
\[
        \|\mathcal B_g-I\|\le \sqrt{2n}\,\delta.
\]
Consequently,
\[
        \|\mathcal B_g\|\le1+\sqrt{2n}\,\delta,
        \qquad
        \ell(\mathcal B_g)\ge1-\sqrt{2n}\,\delta.
\]
\end{lemma}

\begin{proof}
For the identity map, unitary invariance gives
\[
        \mathcal B_{\Id}v
        =2n\int_{S^{2n-1}}\xi\,\Real\langle v,\xi\rangle\,d\sigma(\xi)
        =v.
\]
Thus \(\mathcal B_{\Id}=I\).  If \(\|v\|=1\), then
\[
\begin{aligned}
        \| (\mathcal B_g-I)v\|
        &\le
        2n\delta\int_{S^{2n-1}}
        |\Real\langle v,\xi\rangle|\,d\sigma(\xi)  \\
        &\le
        2n\delta
        \left(
        \int_{S^{2n-1}}(\Real\langle v,\xi\rangle)^2\,d\sigma(\xi)
        \right)^{1/2}.
\end{aligned}
\]
By unitary invariance,
\[
        \int_{S^{2n-1}}(\Real\langle v,\xi\rangle)^2\,d\sigma(\xi)
        =\frac1{2n}.
\]
Therefore \(\|(\mathcal B_g-I)v\|\le\sqrt{2n}\,\delta\).  Taking the supremum
over \(\|v\|=1\) gives the operator norm estimate.  The bounds for
\(\|\mathcal B_g\|\) and \(\ell(\mathcal B_g)\) follow immediately.
\end{proof}

\begin{lemma}[Perturbative control of \(\mathcal A_g\)]\label{lemma-Ag-perturbative}
Let \(g:S^{2n-1}\to S^{2n-1}\) satisfy
\[
        \|g-\Id\|_{L^\infty(S^{2n-1})}\le\delta.
\]
Then
\[
        \|\mathcal A_g-I\|\le2\delta.
\]
Consequently, if \(2\delta<1\), then \(\mathcal A_g\) is invertible and
\[
        \|\mathcal A_g^{-1}\|\le\frac1{1-2\delta},
        \qquad
        \|\mathcal A_g\|\le1+2\delta.
\]
\end{lemma}

\begin{proof}
For the identity map, \(\mathcal A_{\Id}=I\).  Since
\[
        \int_{S^{2n-1}}\xi\langle \xi,u\rangle\,d\sigma(\xi)=0,
\]
we have
\[
\begin{aligned}
        (\mathcal A_g-I)u
        &=-\int_{S^{2n-1}}
        \bigl(g(\xi)\langle g(\xi),u\rangle
        -\xi\langle \xi,u\rangle\bigr)\,d\sigma(\xi).
\end{aligned}
\]
For \(\|u\|=1\),
\[
\begin{aligned}
        \|g\langle g,u\rangle-\xi\langle \xi,u\rangle\|
        &\le
        \|g-\xi\|\,|\langle g,u\rangle|
        +|\langle g-\xi,u\rangle|  \\
        &\le 2\delta.
\end{aligned}
\]
After integration this gives \(\|\mathcal A_g-I\|\le2\delta\).  The remaining
estimates follow from the Neumann lemma.
\end{proof}

\begin{proposition}[Perturbative distortion of the normalized differential]
\label{proposition-perturbative-distortion}
Let \(g:S^{2n-1}\to S^{2n-1}\) be a homeomorphism satisfying
\[
        E_B(g)(0)=0,
\]
and assume
\[
        \|g-\Id\|_{L^\infty(S^{2n-1})}\le\delta.
\]
Put \(s_n=\sqrt{2n}\).  If
\[
        2\delta<1,
        \qquad
        s_n\delta<1,
\]
then \(dE_B(g)_0\) is nonsingular and
\begin{equation}\label{perturbative-distortion-bound}
        K_B(E_B(g),0)
        \le
        \frac{(1+s_n\delta)(1+2\delta)}
        {(1-2\delta)(1-s_n\delta)}.
\end{equation}
In particular,
\[
        K_B(E_B(g),0)\le1+C(n)\delta
\]
for all sufficiently small \(\delta\).
\end{proposition}

\begin{proof}
By Lemma \ref{lemma-perturbative-Bg},
\[
        \|\mathcal B_g-I\|\le s_n\delta.
\]
Hence
\[
        \|\mathcal B_g\|\le1+s_n\delta,
        \qquad
        \ell(\mathcal B_g)\ge1-s_n\delta.
\]
By Lemma \ref{lemma-Ag-perturbative},
\[
        \|\mathcal A_g^{-1}\|\le\frac1{1-2\delta},
        \qquad
        \|\mathcal A_g\|\le1+2\delta.
\]
The normalized first-variation formula gives
\[
        dE_B(g)_0=\mathcal A_g^{-1}\mathcal B_g.
\]
Therefore
\[
        \|dE_B(g)_0\|
        \le
        \frac{1+s_n\delta}{1-2\delta},
\]
and
\[
        \ell(dE_B(g)_0)
        \ge
        \ell(\mathcal A_g^{-1})\ell(\mathcal B_g)
        =
        \frac{\ell(\mathcal B_g)}{\|\mathcal A_g\|}
        \ge
        \frac{1-s_n\delta}{1+2\delta}.
\]
Taking the quotient gives \eqref{perturbative-distortion-bound}.
\end{proof}

The same statement holds with the identity replaced by a unitary map.  Indeed,
if \(U\in U(n)\), then by automorphism naturality
\[
        E_B(U^{-1}\circ g)=U^{-1}\circ E_B(g),
\]
and the left composition by \(U^{-1}\) does not change Bergman singular values
or distortion.

\subsubsection{Completion of the proof}
\label{section-tukia-type-theorem}

We now prove the metric form of the small-distortion theorem.  The proof is the
same normalization principle as in Tukia's theorem: move the point under
consideration to the origin, normalize the image point to the origin, use
compactness to make the normalized boundary map close to a unitary map, and then apply the explicit first-variation estimate.

\begin{proof}[Proof of Theorem~\ref{theorem-tukia-type-almost-isometry}]

Let \(F=E_B(f)\).  Fix \(z_0\in\B^n\).  Choose automorphisms
\(B,A\in\Aut(\B^n)\) such that
\[
        B(0)=z_0,
        \qquad
        A(F(z_0))=0,
\]
and put
\[
        g=A\circ f\circ B|_{S^{2n-1}}.
\]
By naturality,
\[
        E_B(g)=A\circ F\circ B,
        \qquad
        E_B(g)(0)=0.
\]
Since \(A\) and \(B\) are Bergman isometries, the Bergman singular values of
\(dF_{z_0}\) are the same as those of \(dE_B(g)_0\).  Since the CR cross-ratio
is invariant under holomorphic automorphisms and the positive component is
preserved by holomorphic pre- and post-composition, \(g\) is again
CR-orientation-preserving and \(\varepsilon\)-almost CR-M\"obius.

Lemma \ref{lemma-small-distortion-normalization} has the following quantitative
consequence: for every \(\delta>0\) there exists
\(\varepsilon_1=\varepsilon_1(\delta,n)>0\) such that every normalized
CR-orientation-preserving \(\varepsilon\)-almost CR-M\"obius map
\(g\), \(0<\varepsilon<\varepsilon_1\), is within \(\delta\) in the uniform
norm of some unitary map.  Otherwise one could choose
\(\delta_0>0\), \(\varepsilon_j\downarrow0\), and normalized maps \(g_j\)
contradicting Lemma \ref{lemma-small-distortion-normalization}.
After composing \(g\) with such a unitary map, which does not change Bergman
singular values, we may therefore assume
\[
        \|g-\Id\|_{L^\infty(S^{2n-1})}\le \delta,
\]
where \(\delta>0\) can be made arbitrarily small by taking \(\varepsilon\)
sufficiently small.  Proposition
\ref{proposition-perturbative-distortion} and its proof give the two-sided
singular-value estimates
\[
        \|dE_B(g)_0\|
        \le
        \frac{1+\sqrt{2n}\,\delta}{1-2\delta}
        =:L_\delta,
\]
and
\[
        \ell(dE_B(g)_0)
        \ge
        \frac{1-\sqrt{2n}\,\delta}{1+2\delta}
        =:m_\delta.
\]
Thus, for every \(z\in\B^n\),
\begin{equation}\label{global-singular-value-bound}
        m_\delta
        \le
        \ell_{g_B}(dF_z)
        \le
        \|dF_z\|_{g_B}
        \le
        L_\delta.
\end{equation}
Both \(m_\delta\) and \(L_\delta\) tend to \(1\) as \(\delta\to0\).  Therefore,
for the given \(M>1\), choose \(\delta>0\) so small that
\[
        m_\delta\ge M^{-1},
        \qquad
        L_\delta\le M,
\]
and then choose \(\varepsilon_0=\varepsilon_0(M,n)>0\) so that the preceding
normalization estimate holds with this value of \(\delta\) whenever
\(0<\varepsilon<\varepsilon_0\).
Consequently,
\begin{equation}\label{pointwise-M-bound}
        M^{-1}
        \le
        \ell_{g_B}(dF_z)
        \le
        \|dF_z\|_{g_B}
        \le
        M
\end{equation}
for every \(z\in\B^n\).  In particular, \(F\) is a local diffeomorphism.

The map \(F\) extends continuously to \(\overline{\B^n}\) with boundary value
\(f\).  Since \(f\) is a boundary homeomorphism, \(F\) is proper as a map
\(\B^n\to\B^n\): a sequence tending to the boundary is mapped to a sequence
tending to the boundary.  A proper local diffeomorphism has open and closed
image.  Since \(\B^n\) is connected, the image is all of \(\B^n\).  Thus
\(F\) is a covering map of \(\B^n\) onto itself.  Because \(\B^n\) is simply
connected, \(F\) is a global diffeomorphism.  Its interior real-analyticity was
proved in Theorem \ref{smooth-regularity}.

It remains to integrate \eqref{pointwise-M-bound}.  For every piecewise \(C^1\)
curve \(\gamma\),
\[
        \operatorname{length}_B(F\circ\gamma)
        \le
        M\operatorname{length}_B(\gamma).
\]
Taking the infimum over curves joining \(x\) to \(y\) gives
\[
        d_B(F(x),F(y))\le M d_B(x,y).
\]
Since \(F\) is a diffeomorphism, \eqref{pointwise-M-bound} also gives
\(\|dF^{-1}\|_{g_B}\le M\).  Applying the preceding estimate to \(F^{-1}\)
yields
\[
        d_B(x,y)
        \le
        M d_B(F(x),F(y)),
\]
which is the lower bound in \eqref{tukia-type-distance-estimate}.  The
quasiconformal estimate \(K_B(F)\le M^2\) follows from the pointwise singular
value bounds.
\end{proof}

\begin{corollary}[Small CR-quasisymmetry]
If \(f:S^{2n-1}\to S^{2n-1}\) is CR-orientation-preserving and has sufficiently small CR-quasisymmetric
distortion with respect to \(d_{CR}\), then \(E_B(f)\) is a Bergman
bi-Lipschitz diffeomorphism.  More precisely, for every \(M>1\), if the map is CR-orientation-preserving and the
CR-quasisymmetric distortion is sufficiently close to \(1\), then
\[
        M^{-1}d_B(x,y)
        \le d_B(E_B(f)(x),E_B(f)(y))
        \le M d_B(x,y).
\]
\end{corollary}

\begin{proof}
On the compact CR sphere, small quasisymmetric distortion implies small
quasi-M\"obius distortion quantitatively.  The claim follows from Theorem
\ref{theorem-tukia-type-almost-isometry}.
\end{proof}

\section{Counterexample and sharpness}
\label{section-cr-qs-counterexample}

\subsection{Anti-holomorphic symmetry}

\begin{lemma}[Anti-holomorphic symmetry of barycenters]
\label{lemma-antiholomorphic-barycenter-symmetry}
Let \(C(z)=\overline z\) be complex conjugation on \(\B^n\).  Then, for every
admissible full-support probability measure \(\mu\) on \(S^{2n-1}\),
\[
        \barB(C_*\mu)=C(\barB(\mu)).
\]
More generally, the same conclusion holds for every anti-holomorphic Bergman
isometry obtained by composing complex conjugation with a holomorphic
automorphism.
\end{lemma}

\begin{proof}
For \(\eta\in S^{2n-1}\) and \(z\in\B^n\), the formula
\[
        \beta_\eta(z)
        =\log\frac{|1-\langle z,\eta\rangle|^2}{1-|z|^2}
\]
gives
\[
        \beta_{C\eta}(Cz)=\beta_\eta(z),
\]
because \(\langle Cz,C\eta\rangle=\overline{\langle z,\eta\rangle}\) and
\(|Cz|=|z|\).  Hence
\[
        \mathcal B_{C_*\mu}(Cz)
        =\int_{S^{2n-1}}\beta_{C\eta}(Cz)\,d\mu(\eta)
        =\mathcal B_\mu(z).
\]
Thus \(z\) minimizes \(\mathcal B_\mu\) if and only if \(Cz\) minimizes
\(\mathcal B_{C_*\mu}\).  Uniqueness of barycenters gives the result.  The
last statement follows by combining this identity with the holomorphic
equivariance already proved in Proposition~\ref{equivariance-barycenter}.
\end{proof}

\subsection{The contact tennis-ball construction}

We now show that the non-perturbative normalized non-collapse statement fails
in the full CR-quasisymmetric class.  The point is that the tennis-ball type
construction, in the spirit of Laugesen's construction \cite{Laugesen1996},
can be carried out inside the contact category.  Thus the boundary maps
constructed below are not merely round quasisymmetric; they are quasisymmetric
for the CR visual metric.

We use the following notation.  Write a point of \(S^{2n-1}\) as
\[
        \xi=(w,\sqrt{1-|w|^2}\,\omega),
        \qquad
        w\in\overline{\D},\quad \omega\in S^{2n-3}.
\]
On \(\D\) put
\[
        \alpha_{\D}
        =
        \frac{\operatorname{Im}(\overline w\,dw)}{1-|w|^2},
        \qquad
        \Omega_{\D}=d\alpha_{\D}.
\]
Thus, for \(w=re^{i\theta}\),
\[
        \Omega_{\D}
        =
        \frac{2r}{(1-r^2)^2}\,dr\wedge d\theta.
\]
The form \(\Omega_{\D}\) has infinite total area near \(\partial\D\), but finite
area on compact subsets of \(\D\).

Let \(\theta\) denote the standard contact form on \(S^{2n-1}\).  We write
\[
        \mathcal H_\xi=\ker\theta_\xi
\]
for the standard contact distribution.  Explicitly, if
\(V\in T_\xi S^{2n-1}\), then
\[
        \theta_\xi(V)=\operatorname{Im}\langle V,\xi\rangle,
\]
and hence
\[
        \mathcal H_\xi
        =
        \left\{
        V\in T_\xi S^{2n-1}:
        \operatorname{Im}\langle V,\xi\rangle=0
        \right\}.
\]
Since
\[
        T_\xi S^{2n-1}
        =
        \left\{
        V\in\C^n:
        \operatorname{Re}\langle V,\xi\rangle=0
        \right\},
\]
this is equivalently
\[
        \mathcal H_\xi
        =
        \left\{
        V\in\C^n:
        \langle V,\xi\rangle=0
        \right\}.
\]

Let \(d_{CC}\) be the Carnot--Carath\'eodory distance associated with
\(\mathcal H\) and the restriction of the round metric to \(\mathcal H\).
Thus
\[
        d_{CC}(\xi,\zeta)
        =
        \inf_\gamma
        \int_0^1 |\gamma'(t)|\,dt,
\]
where the infimum is taken over all piecewise \(C^1\) curves
\(\gamma:[0,1]\to S^{2n-1}\) satisfying
\[
        \gamma(0)=\xi,\qquad \gamma(1)=\zeta,
        \qquad
        \gamma'(t)\in\mathcal H_{\gamma(t)}
        \quad\text{for a.e. }t.
\]
Equivalently,
\[
        \langle \gamma'(t),\gamma(t)\rangle=0
        \quad\text{for a.e. }t.
\]
Here \(|\gamma'(t)|\) is the Euclidean norm inherited from \(\C^n\).

We shall also use the CR visual metric
\[
        d_{CR}(\xi,\zeta)
        =
        |1-\langle \xi,\zeta\rangle|^{1/2}.
\]
The metrics \(d_{CC}\) and \(d_{CR}\) are bi-Lipschitz equivalent on
\(S^{2n-1}\); namely, there exists a constant \(C_n\ge1\), depending only on
\(n\), such that
\[
        C_n^{-1}d_{CC}(\xi,\zeta)
        \le
        d_{CR}(\xi,\zeta)
        \le
        C_n d_{CC}(\xi,\zeta)
\]
for all \(\xi,\zeta\in S^{2n-1}\).  If \(h\) is a \(C^1\) contact
diffeomorphism, we write
\[
        d_{\mathcal H}h_\xi
        :=
        dh_\xi|_{\mathcal H_\xi}:
        \mathcal H_\xi\to\mathcal H_{h(\xi)}
\]
for its horizontal differential.

\begin{lemma}[Quantitative quasisymmetry of contact diffeomorphisms]
\label{lemma-quantitative-qs-contact}
Let \(h:S^{2n-1}\to S^{2n-1}\) be a smooth contact diffeomorphism.  Define
\[
        L_h
        =
        \sup_{\xi\in S^{2n-1}}
        \|d_{\mathcal H}h_\xi\|,
        \qquad
        L_{h^{-1}}
        =
        \sup_{\xi\in S^{2n-1}}
        \|d_{\mathcal H}h^{-1}_\xi\|.
\]
Then \(h\) is quasisymmetric with respect to \(d_{CC}\), with distortion
function
\[
        \eta_{CC}(t)=L_hL_{h^{-1}}t.
\]
Consequently \(h\) is quasisymmetric with respect to \(d_{CR}\), with distortion
function
\[
        \eta_{CR}(t)=C_n^4L_hL_{h^{-1}}t.
\]
\end{lemma}

\begin{proof}
Since \(h\) is contact, it maps horizontal curves to horizontal curves.  If
\(\gamma\) is horizontal, then
\[
        \operatorname{length}_{CC}(h\circ\gamma)
        \le
        L_h\,\operatorname{length}_{CC}(\gamma).
\]
Taking the infimum over horizontal curves joining \(\xi\) to \(\eta\), we get
\[
        d_{CC}(h(\xi),h(\eta))
        \le
        L_h\,d_{CC}(\xi,\eta).
\]
Applying the same argument to \(h^{-1}\) gives
\[
        d_{CC}(\xi,\eta)
        \le
        L_{h^{-1}}\,d_{CC}(h(\xi),h(\eta)).
\]
Therefore
\[
        \frac{d_{CC}(h(\xi),h(\eta))}
             {d_{CC}(h(\xi),h(\alpha))}
        \le
        L_hL_{h^{-1}}
        \frac{d_{CC}(\xi,\eta)}
             {d_{CC}(\xi,\alpha)}.
\]
This proves the \(d_{CC}\)-quasisymmetry statement.  Using
\[
        C_n^{-1}d_{CC}\le d_{CR}\le C_n d_{CC},
\]
we obtain
\[
\begin{aligned}
        \frac{d_{CR}(h(\xi),h(\eta))}
             {d_{CR}(h(\xi),h(\alpha))}
        &\le
        C_n^2
        \frac{d_{CC}(h(\xi),h(\eta))}
             {d_{CC}(h(\xi),h(\alpha))}       \\
        &\le
        C_n^2 L_hL_{h^{-1}}
        \frac{d_{CC}(\xi,\eta)}
             {d_{CC}(\xi,\alpha)}             \\
        &\le
        C_n^4L_hL_{h^{-1}}
        \frac{d_{CR}(\xi,\eta)}
             {d_{CR}(\xi,\alpha)}.
\end{aligned}
\]
Hence \(h\) is \(d_{CR}\)-quasisymmetric with distortion function
\[
        \eta_{CR}(t)=C_n^4L_hL_{h^{-1}}t.
\]
\end{proof}

\begin{lemma}[Contact lift of an area-preserving disk map]
\label{lemma-contact-lift}
Let \(H:\D\to\D\) be a smooth orientation-preserving diffeomorphism satisfying
\[
        H^*\Omega_{\D}=\Omega_{\D}.
\]
Assume that \(H\) extends smoothly to \(\overline{\D}\), is the identity in a
collar of \(\partial\D\), and is conjugation-equivariant:
\[
        H(\overline w)=\overline{H(w)}.
\]
Then there exists a smooth real-valued function
\[
        \psi:\overline{\D}\to\R,
        \qquad
        \psi(\overline w)=-\psi(w),
\]
such that
\[
        h(w,\sqrt{1-|w|^2}\,\omega)
        =
        \left(
        H(w),
        \sqrt{1-|H(w)|^2}\,e^{i\psi(w)}\omega
        \right)
\]
defines a smooth contact diffeomorphism of \(S^{2n-1}\).  Moreover, \(h\)
commutes with the maps
\[
        R_U(z_1,z')=(z_1,Uz'),
        \qquad U\in U(n-1),
\]
and with complex conjugation
\[
        C(z_1,z')=(\overline z_1,\overline{z'}).
\]
If \(H\) is the identity in neighborhoods of \(1\) and \(-1\), then \(h\) is
the identity in neighborhoods of
\[
        N=e_1,
        \qquad
        S=-e_1.
\]
\end{lemma}

\begin{proof}
On \(S^{2n-1}\), the standard contact form may be written, up to a harmless
constant factor, as
\[
        \theta
        =
        \operatorname{Im}(\overline w\,dw)
        +
        (1-|w|^2)\operatorname{Im}\langle\omega,d\omega\rangle.
\]
For the proposed map \(h\),
\[
\begin{aligned}
        h^*\theta
        &=
        \operatorname{Im}(\overline{H(w)}\,dH(w))
        +
        (1-|H(w)|^2)d\psi(w)  \\
        &\qquad
        +
        (1-|H(w)|^2)\operatorname{Im}\langle\omega,d\omega\rangle.
\end{aligned}
\]
Thus \(h\) is contact if
\[
        h^*\theta
        =
        \frac{1-|H(w)|^2}{1-|w|^2}\,\theta.
\]
This is equivalent to
\[
        d\psi
        =
        \frac{\operatorname{Im}(\overline w\,dw)}{1-|w|^2}
        -
        \frac{\operatorname{Im}(\overline{H(w)}\,dH(w))}{1-|H(w)|^2}
        =
        \alpha_{\D}-H^*\alpha_{\D}.
\]
Since
\[
        d(\alpha_{\D}-H^*\alpha_{\D})
        =
        \Omega_{\D}-H^*\Omega_{\D}
        =
        0,
\]
and \(\D\) is simply connected, the one-form
\(\alpha_{\D}-H^*\alpha_{\D}\) is exact.  Because \(H=\Id\) in a collar of
\(\partial\D\), this one-form vanishes there; therefore its primitive extends
smoothly to \(\overline{\D}\).  We choose this primitive as \(\psi\).

Let
\[
        \kappa:\D\to\D,
        \qquad
        \kappa(w)=\overline w,
\]
be complex conjugation.  Since
\[
        \alpha_{\D}
        =
        \frac{\operatorname{Im}(\overline w\,dw)}{1-|w|^2},
\]
we have
\[
        \kappa^*\alpha_{\D}=-\alpha_{\D}.
\]
Indeed, if \(w=x+iy\), then
\[
        \operatorname{Im}(\overline w\,dw)=x\,dy-y\,dx,
\]
and under \(\kappa(x,y)=(x,-y)\) this one-form changes sign, while the
denominator \(1-|w|^2\) is unchanged.  Since \(H\) is
conjugation-equivariant,
\[
        H\circ\kappa=\kappa\circ H,
\]
and hence
\[
        \kappa^*(\alpha_{\D}-H^*\alpha_{\D})
        =
        -(\alpha_{\D}-H^*\alpha_{\D}).
\]
Therefore, if
\[
        d\psi=\alpha_{\D}-H^*\alpha_{\D},
\]
then
\[
        d(\psi\circ\kappa)=-d\psi.
\]
Thus \(\psi\circ\kappa+\psi\) is constant.  After changing \(\psi\) by an
additive constant, we may assume
\[
        \psi(\overline w)=-\psi(w).
\]
The assertions about the \(U(n-1)\)-symmetry and complex conjugation follow directly from the formula for \(h\).  If \(H=\Id\) near \(1\) and \(-1\), then
\(\psi\) is locally constant there; after the same choice of additive constant,
\(h\) is the identity near \(N\) and \(S\).
\end{proof}

\begin{lemma}[Area-preserving contact tennis-ball map]
\label{lemma-contact-tennis-ball}
Let \(n\ge2\), and fix any \(t_0\in(0,1)\).  Then there exists a smooth
contact diffeomorphism
\[
        h:S^{2n-1}\to S^{2n-1}
\]
with the following properties.  Put
\[
        N=e_1,
        \qquad
        S=-e_1,
        \qquad
        z_t=te_1,
        \qquad
        \sigma_t=\sigma_{z_t},
        \qquad
        x(\xi)=\operatorname{Re}\xi_1.
\]
For \(U\in U(n-1)\), write
\[
        R_U(z_1,z')=(z_1,Uz'),
\]
and let
\[
        C(z_1,z')=(\overline z_1,\overline{z'}).
\]
Then:
\begin{enumerate}
\item \(h\circ R_U=R_U\circ h\) for every \(U\in U(n-1)\).
\item \(h\circ C=C\circ h\).
\item \(h\) is the identity near \(N\) and \(S\).
\item \(h\) is CR-quasisymmetric.  More precisely, it is
\(\eta\)-quasisymmetric with
\[
        \eta(t)=C_n^4L_hL_{h^{-1}}t,
\]
where \(L_h\) and \(L_{h^{-1}}\) are the horizontal Lipschitz constants from
Lemma~\ref{lemma-quantitative-qs-contact}.
\item For the prescribed \(t_0\in(0,1)\),
\[
        \int_{S^{2n-1}} x(h(\xi))\,d\sigma_{t_0}(\xi)<0
\]
and
\[
        \int_{S^{2n-1}} x(h(\xi))\,d\sigma_{-t_0}(\xi)>0.
\]
\end{enumerate}
\end{lemma}

\begin{proof}
Let \(\nu_t\) denote the projection of \(\sigma_{te_1}\) to the closed
\(w\)-disk.  Its density has the form
\[
        c_n
        \left(
        \frac{1-t^2}{|1-t\overline w|^2}
        \right)^n
        (1-|w|^2)^{n-2},
        \qquad |w|<1.
\]
For every \(t>0\),
\[
        \nu_t(\{\operatorname{Re}w>0\})>\frac12,
\]
and
\[
        \nu_{-t}(\{\operatorname{Re}w<0\})>\frac12.
\]
Indeed, if \(w=x+iy\) with \(x>0\), then
\[
        |1-t\overline w|^2
        =
        1-2tx+t^2|w|^2
        <
        1+2tx+t^2|w|^2
        =
        |1+t\overline w|^2.
\]
Thus the density of \(\nu_t\) at \(w\) is larger than the density at \(-w\),
and the first strict inequality follows by pairing the two half-disks.  The
second inequality follows by reflection.  In particular, the prescribed
\(t_0\) may be chosen arbitrarily close to \(0\), although the bias then becomes
small.

Choose a number \(p\) such that
\[
        \frac12<p<
        \min\left\{
        \nu_{t_0}(\{\operatorname{Re}w>0\}),
        \nu_{-t_0}(\{\operatorname{Re}w<0\})
        \right\}.
\]
Then choose \(a\in(0,1)\) so close to \(1\) that
\[
        p>\frac1{1+a}.
\]
The real axis has \(\nu_{t_0}\)-measure zero.  Hence, by inner regularity,
we may choose finitely many pairwise disjoint closed smooth disks
\[
        K_{+,1},\ldots,K_{+,m}
        \Subset
        \{w\in\D:\operatorname{Re}w>0,\ \operatorname{Im}w>0\}
\]
such that, with
\[
        K_+
        =
        \bigcup_{j=1}^m
        \bigl(K_{+,j}\cup\overline{K_{+,j}}\bigr),
        \qquad
        K_-=-K_+,
\]
one has
\[
        \nu_{t_0}(K_+)>p,
        \qquad
        \nu_{-t_0}(K_-)>p.
\]
Here \(\overline{K_{+,j}}=\{\overline w:w\in K_{+,j}\}\).  Thus
\(K_+\) and \(K_-\) are conjugation-invariant compact sets, but their
components stay a positive distance from the real axis.

We next choose target disks.  Since
\[
        \Omega_{\D}
        =
        \frac{2r}{(1-r^2)^2}\,dr\wedge d\theta,
\]
the \(\Omega_{\D}\)-area of each of the upper-half truncated caps
\[
        \{w\in\D:\operatorname{Re}w>a,\ \operatorname{Im}w>0,
        \ |w|<1-\delta\}
\]
and
\[
        \{w\in\D:\operatorname{Re}w<-a,\ \operatorname{Im}w>0,
        \ |w|<1-\delta\}
\]
tends to \(+\infty\) as \(\delta\downarrow0\).  After choosing
\(\delta>0\) sufficiently small, choose pairwise disjoint closed smooth disks
\[
        Q_{-,j}
        \Subset
        \{\operatorname{Re}w<-a,\ \operatorname{Im}w>0,
        \ |w|<1-\delta\},
\]
\[
        Q_{+,j}
        \Subset
        \{\operatorname{Re}w>a,\ \operatorname{Im}w>0,
        \ |w|<1-\delta\},
        \qquad 1\le j\le m,
\]
so that the \(\Omega_{\D}\)-area of \(Q_{-,j}\) equals that of
\(K_{+,j}\), while the \(\Omega_{\D}\)-area of \(Q_{+,j}\) equals that of
\(-\overline{K_{+,j}}\).  Put
\[
        Q_-
        =
        \bigcup_{j=1}^m
        \bigl(Q_{-,j}\cup\overline{Q_{-,j}}\bigr),
        \qquad
        Q_+
        =
        \bigcup_{j=1}^m
        \bigl(Q_{+,j}\cup\overline{Q_{+,j}}\bigr).
\]
Then
\[
        Q_-\Subset\{\operatorname{Re}w<-a,\ |w|<1-\delta\},
        \qquad
        Q_+\Subset\{\operatorname{Re}w>a,\ |w|<1-\delta\}.
\]

Choose the collar
\[
        \mathcal C_\delta=\{w\in\D:1-\delta/2<|w|<1\}.
\]
All the source and target disks are disjoint from \(\mathcal C_\delta\) and
from the real axis.  On the upper half-disk, the relative Moser theorem,
equivalently the compactly supported Hamiltonian isotopy extension theorem,
gives an \(\Omega_{\D}\)-area-preserving diffeomorphism, supported in a
compact set away from both the real axis and \(\mathcal C_\delta\), which
maps
\[
        K_{+,j}\ \text{into}\ Q_{-,j},
        \qquad
        -\overline{K_{+,j}}\ \text{into}\ Q_{+,j},
        \qquad 1\le j\le m.
\]
Reflect this diffeomorphism across the real axis and extend it by the identity
near the real axis and on \(\mathcal C_\delta\).  The resulting map
\[
        H:\D\to\D
\]
is a smooth orientation-preserving \(\Omega_{\D}\)-area-preserving
diffeomorphism satisfying
\[
        H(K_+)\subset Q_-,
        \qquad
        H(K_-)\subset Q_+,
\]
\[
        H(\overline w)=\overline{H(w)},
\]
and \(H=\Id\) on \(\mathcal C_\delta\).

Apply Lemma~\ref{lemma-contact-lift}.  We obtain a smooth contact
diffeomorphism
\[
        h(w,\sqrt{1-|w|^2}\,\omega)
        =
        \left(
        H(w),
        \sqrt{1-|H(w)|^2}\,e^{i\psi(w)}\omega
        \right).
\]
It has the asserted symmetries and is the identity near \(N\) and \(S\).
By Lemma~\ref{lemma-quantitative-qs-contact}, it is CR-quasisymmetric with the
stated quantitative distortion function.

Finally,
\[
        x(h(\xi))=\operatorname{Re}H(w).
\]
Since \(H(K_+)\subset Q_-\subset\{\operatorname{Re}w<-a\}\), we get
\[
\begin{aligned}
        \int_{S^{2n-1}}x(h(\xi))\,d\sigma_{t_0}(\xi)
        &=
        \int_{\overline{\D}}\operatorname{Re}H(w)\,d\nu_{t_0}(w)  \\
        &\le
        -a\,\nu_{t_0}(K_+)
        +
        \bigl(1-\nu_{t_0}(K_+)\bigr)  \\
        &=
        1-(1+a)\nu_{t_0}(K_+)  \\
        &<
        1-(1+a)p
        <
        0.
\end{aligned}
\]
Similarly,
\[
\begin{aligned}
        \int_{S^{2n-1}}x(h(\xi))\,d\sigma_{-t_0}(\xi)
        &=
        \int_{\overline{\D}}\operatorname{Re}H(w)\,d\nu_{-t_0}(w)  \\
        &\ge
        a\,\nu_{-t_0}(K_-)
        -
        \bigl(1-\nu_{-t_0}(K_-)\bigr)  \\
        &=
        (1+a)\nu_{-t_0}(K_-)-1  \\
        &>
        (1+a)p-1
        >
        0.
\end{aligned}
\]
This proves the lemma.
\end{proof}

\subsection{Failure of injectivity}

\begin{proof}[Proof of Theorem~\ref{theorem-cr-qs-counterexample}]

Let \(h\) be the contact tennis-ball map from
Lemma~\ref{lemma-contact-tennis-ball}.  Since \(h\) is a smooth contact
diffeomorphism of the compact CR sphere, it is CR-quasisymmetric and
CR-orientation-preserving.

For \(t\in(-1,1)\), put \(z_t=te_1\).  The visual measure \(\sigma_{z_t}\) is
invariant under \(U(n-1)\) and under complex conjugation.  Since \(h\) commutes
with these symmetries, the measure
\[
        h_*\sigma_{z_t}
\]
has the same symmetries.  Barycenters are equivariant under holomorphic
automorphisms and under complex conjugation.  Hence
\[
        E_B(h)(te_1)=u(t)e_1
\]
for some continuous real-valued function \(u:(-1,1)\to(-1,1)\).

Let \(\mu\) be a probability measure on \(S^{2n-1}\) with the same symmetries,
and suppose
\[
        \barB(\mu)=se_1.
\]
Along the real geodesic \(s\mapsto se_1\),
\[
        \frac{d}{ds}\bigg|_{s=0}\beta_\eta(se_1)
        =
        -2\operatorname{Re}\eta_1.
\]
Therefore the derivative at \(0\) of the averaged Busemann potential is
\[
        -2\int_{S^{2n-1}}\operatorname{Re}\eta_1\,d\mu(\eta).
\]
Since the averaged Busemann potential is strictly convex on this real geodesic,
its minimizer lies on the negative side if this derivative is positive, and on
the positive side if this derivative is negative.

Apply this to
\[
        \mu_+=h_*\sigma_{t_0e_1},
        \qquad
        \mu_-=h_*\sigma_{-t_0e_1}.
\]
By Lemma~\ref{lemma-contact-tennis-ball},
\[
        \int \operatorname{Re}\eta_1\,d\mu_+(\eta)<0,
        \qquad
        \int \operatorname{Re}\eta_1\,d\mu_-(\eta)>0.
\]
Hence
\[
        u(t_0)<0,
        \qquad
        u(-t_0)>0.
\]
On the other hand, \(h\) is the identity near \(N=e_1\) and \(S=-e_1\).  The
boundary extension theorem gives
\[
        \lim_{t\to1}u(t)=1,
        \qquad
        \lim_{t\to-1}u(t)=-1.
\]
By continuity, there exist
\[
        a\in(-1,-t_0),
        \qquad
        b\in(t_0,1)
\]
such that
\[
        u(a)=0=u(b).
\]
Therefore
\[
        E_B(h)(ae_1)=0=E_B(h)(be_1),
        \qquad a\ne b.
\]
Thus \(E_B(h)\) is not injective.
\end{proof}

\begin{corollary}[Failure of normalized non-collapse]
\label{corollary-failure-noncollapse}
For every \(n\ge2\), there exists a CR-orientation-preserving
CR-quasisymmetric homeomorphism
\[
        g:S^{2n-1}\to S^{2n-1}
\]
such that
\[
        E_B(g)(0)=0
\]
and
\[
        dE_B(g)_0
\]
is singular.  Equivalently,
\[
        \ell(dE_B(g)_0)=0.
\]
Thus the normalized matrix non-collapse statement has a negative answer in the
full CR-quasisymmetric class.
\end{corollary}

\begin{proof}
Let \(h\) be the CR-quasisymmetric map from
Theorem~\ref{theorem-cr-qs-counterexample}.  Since \(E_B(h)\) extends
continuously to the closed ball with boundary value \(h\), it is proper.

If \(dE_B(h)\) were nonsingular at every point of \(\B^n\), then \(E_B(h)\)
would be a proper local diffeomorphism
\[
        \B^n\to\B^n.
\]
Since \(\B^n\) is simply connected, \(E_B(h)\) would be a global
diffeomorphism, contradicting Theorem~\ref{theorem-cr-qs-counterexample}.
Therefore there exists \(z_0\in\B^n\) such that
\[
        dE_B(h)_{z_0}
\]
is singular.

Choose automorphisms \(B,A\in\Aut(\B^n)\) satisfying
\[
        B(0)=z_0,
        \qquad
        A(E_B(h)(z_0))=0.
\]
Define
\[
        g=A\circ h\circ B|_{S^{2n-1}}.
\]
Since holomorphic automorphisms act by CR-M\"obius transformations on the
boundary, \(g\) is still CR-orientation-preserving and CR-quasisymmetric.  By
naturality,
\[
        E_B(g)=A\circ E_B(h)\circ B,
        \qquad
        E_B(g)(0)=0.
\]
Moreover,
\[
        dE_B(g)_0
        =
        dA_{E_B(h)(z_0)}
        \circ
        dE_B(h)_{z_0}
        \circ
        dB_0.
\]
The first and third factors are invertible, while the middle factor is
singular.  Hence \(dE_B(g)_0\) is singular.
\end{proof}

\subsection{Failure of non-perturbative local non-collapse}
\label{section-failure-nonperturbative}

Theorem~\ref{theorem-general-quasi-isometry} gives non-perturbative
large-scale control for the full CR-quasisymmetric class.  The counterexample
above shows that this conclusion cannot in general be upgraded to local
non-degeneracy or to a diffeomorphism theorem.

Indeed, Corollary~\ref{corollary-failure-noncollapse} gives a
CR-orientation-preserving CR-quasisymmetric boundary homeomorphism
\(g:S^{2n-1}\to S^{2n-1}\) such that
\[
        E_B(g)(0)=0,
        \qquad dE_B(g)_0\text{ is singular}.
\]
Equivalently, in the normalized first-variation formula
\[
        dE_B(g)_0=\mathcal A_g^{-1}\mathcal B_g
\]
one has
\[
        \ell(\mathcal A_g^{-1}\mathcal B_g)=0.
\]
If \(\eta_0\) is a CR-quasisymmetric distortion function for this map, then no
estimate of the form
\[
        \ell(dE_B(f)_0)\ge c(\eta_0,n)>0
\]
can hold for all normalized \(\eta_0\)-CR-quasisymmetric boundary
homeomorphisms.  Thus fixed CR-quasisymmetric distortion controls the extension
coarsely but does not control the smallest singular value of its differential.

There is no contradiction between the two conclusions.  A quasi-isometry may
identify points at uniformly bounded distance and may have a singular
differential at individual points.  Small CR cross-ratio distortion is much
more rigid: after automorphic normalization the boundary map is uniformly close
to a unitary map, and the first-variation formula then forces the differential
to be uniformly close to an isometry.  This is exactly the additional input
which yields the bi-Lipschitz diffeomorphism theorem.

\subsection*{Concluding remarks}

The Bergman ball provides a natural complex-hyperbolic setting for the
barycentric philosophy behind conformally natural extensions.  The extension
constructed here is canonical and automorphism-natural, agrees with
holomorphic automorphisms, has the prescribed boundary values, and is
real-analytic in the interior.

For arbitrary CR-quasisymmetric boundary data, the normalized quasi-M\"obius
compactness mechanism gives a global large-scale conclusion: \(E_B(f)\) is a
quasi-isometry and \(E_B(f^{-1})\) is its coarse inverse.  Under small positive
CR cross-ratio distortion, normalized boundary maps are uniformly close to
unitary maps; the explicit formula
\[
        dE_B(g)_0=\mathcal A_g^{-1}\mathcal B_g
\]
then upgrades coarse control to a Bergman bi-Lipschitz real-analytic
diffeomorphism, with bi-Lipschitz constant tending to \(1\) as the distortion
tends to zero.

The contact tennis-ball construction shows that this upgrade is genuinely
perturbative.  In the full CR-quasisymmetric class the canonical extension may
be non-injective and may have singular differential, while remaining a
quasi-isometry.  Thus the large-scale extension theorem is non-perturbative,
but local non-degeneracy is not.

\subsection*{Declaration on the use of AI-assisted tools}
During the preparation of this manuscript, the authors used OpenAI's ChatGPT
(Pro subscription) to assist with language editing, stylistic refinement, and
as a discussion aid for checking the formulation and consistency of certain
ideas. All mathematical statements, proofs, computations, references, and
conclusions were independently reviewed and verified by the authors. The
authors take full responsibility for the content of the manuscript.

\subsection*{Funding}
The authors are partially supported by the Ministry of Education, Science and Innovation of Montenegro through the grants \emph{Mathematical Analysis, Optimization and Machine Learning} and \emph{Complex-analytic and geometric techniques for non-Euclidean machine learning: theory and applications}.

\bigskip
\noindent\textsc{David Kalaj}\\
Faculty of Natural Sciences and Mathematics, University of Montenegro, Podgorica, Montenegro\\
\emph{Email address:} \href{mailto:davidkalaj@gmail.com}{davidkalaj@gmail.com}

\medskip
\noindent\textsc{Anton Gjokaj}\\
Faculty of Natural Sciences and Mathematics, University of Montenegro, Podgorica, Montenegro\\
\emph{Email address:} \href{mailto:antongjpmf@gmail.com}{antongjpmf@gmail.com}

\medskip
\noindent\textsc{Vladimir Ja\'cimovi\'c}\\
Faculty of Natural Sciences and Mathematics, University of Montenegro, Podgorica, Montenegro\\
\emph{Email address:} \href{mailto:vladimir@jacimovic.me}{vladimir@jacimovic.me}

\end{document}